\documentclass[a4paper, 10pt, twoside, notitlepage, reqno]{amsart}
\usepackage{amsaddr}
\usepackage[utf8x]{inputenc}
\usepackage{color}
\usepackage{amsmath}
\usepackage{amssymb}
\usepackage{amsthm}
\usepackage{graphicx}
\usepackage{esint}
\usepackage[colorlinks=true,linkcolor=blue]{hyperref}
\usepackage{verbatim}
\usepackage{geometry}

\theoremstyle{plain}
\newtheorem{thm}{Theorem}
\newtheorem{prop}{Proposition}[section]

\newtheorem{cor}[prop]{Corollary}

\newtheorem{defi}[prop]{Definition}
\newtheorem{rmk}[prop]{Remark}

\newcommand {\R} {\mathbb{R}}

\newcommand {\p} {\partial}

\newcommand {\supp} {\text{supp}}

\usepackage{xspace}

\title{The gradient flow of the potential energy on the space of arcs}
\author{Wenhui Shi \& Dmitry Vorotnikov}%
\address[D.~Vorotnikov]{CMUC, Department of
Mathematics, University of Coimbra, 3001-501 Coimbra, Portugal}{}
\email{wshi@mat.uc.pt, mitvorot@mat.uc.pt}

\begin{document}
\begin{abstract} We study the gradient flow of the potential energy on the infinite-dimensional Riemannian manifold of spatial curves parametrized by the arc length, which models overdamped motion of a falling inextensible string. We prove existence of generalized solutions to the corresponding nonlinear evolutionary PDE and their exponential decay to the equilibrium. We also observe that the system admits solutions backwards in time, which leads to non-uniqueness of trajectories.
\end{abstract}
\maketitle


\vspace{10pt}


\section{Introduction}

Consider a regular arc of unit length in $\R^d$ with one end fixed at zero. By natural parametrization, it can be identified with a map $\eta:[0,1]\to \R^d$ subject to the constraints $$\eta(1)=0,\ |\p_s\eta(s)|=1 \text{ for each }s\in [0,1].$$
The space of all arcs can thus be at least formally viewed as an infinite-dimensional submanifold
\begin{align*}
\mathcal{A}=\{\eta\in H^2(0,1;\R^d): \eta(1)=0,\ |\p_s\eta(s)|=1 \text{ for all }s\in [0,1]\}
\end{align*}
of the Hilbert space $H^2(0,1;\R^d)$. The (weak) Riemannian metric on $H^2(0,1;\R^d)$ given by the $L^2$ inner product restricts to each tangent space $T_\eta\mathcal{A}$, $\eta\in \mathcal{A}$, i.e.
$$\left\langle \cdot, \cdot \right\rangle_{T_{\eta}\mathcal{A}}=\left\langle \cdot, \cdot \right\rangle_{L^2(0,1;\R^d)},$$
and thus $\mathcal{A}$ becomes a Riemannian submanifold. Here the tangent space $T_\eta\mathcal{A}$ is the vector space consisting of all vector fields $v\in H^2(0,1;\R^d)$ which satisfy the compatibility conditions \begin{align} \label{e:compcons}
v(1)=0,\ \p_s\eta(s)\cdot \p_s v(s)=0\text{ for each }s\in [0,1].
\end{align}
Note that we endow $\mathcal{A}$ with the $L^2$ inner product and not with the Hilbertian inner product on $H^2$. We refer to \cite{Preston-2012} for a rigorous treatment of the geometric properties of this manifold of arcs. Various geometries on spaces of curves are discussed in the review papers \cite{MM07,BBM14}.  The geodesics on $\mathcal{A}$ are the solutions to the PDE system \begin{equation}\label{eq:geod}
\begin{split}
\begin{cases}
\p_{tt}\eta(t,s) = \p_s(\varsigma(t,s) \p_s\eta(t,s)), \\
|\p_s\eta(t,s)|  =1
\end{cases}
\end{split}
\end{equation} subject to the boundary conditions \begin{equation} \label{e:bkg}
\eta(t,1)=0, \quad \varsigma (t,0)=0.
\end{equation} The auxiliary unknown scalar function $\varsigma$ can be viewed as the Lagrange multiplier coming from the pointwise constraint $|\p_s\eta|=1$.
These equations have been known in mechanics since very old times (and thus much before any geodesic interpretation). They are the equations of motion of inextensible strings and can be used to describe the dynamics of whips, chains, flagella etc.

In the presence of a constant gravity $g\in \R^d$, the motion equations for the inextensible string become \begin{equation}\label{eq:geodgrav}
\begin{split}
\begin{cases}
\p_{tt}\eta(t,s) = \p_s(\varsigma(t,s) \p_s\eta(t,s))+g, \\
|\p_s\eta(t,s)| =1.
\end{cases}
\end{split}
\end{equation} The celebrated catenary problem of finding the shape of a stationary chain with two fixed ends corresponds to the time-independent solution to \eqref{eq:geodgrav}. Galileo \cite{RenEl} thought that  it was a parabola, but Johann Bernoulli, Leibniz and Huygens proved that it was a hyperbolic cosine.
System \eqref{eq:geodgrav}, \eqref{e:bkg} can be deemed as a manifestation of the physical principle of least action \cite{F64} for the Lagrangian action $$\int_{t_0}^{t_1} \frac 1 2 \|\p_t \eta(t)\|^2_{L^2}-E(\eta(t))\,dt,\ \eta(t) \in \mathcal{A},$$ where \begin{align*}
E(\eta):=\int_0^1(-g)\cdot \eta
\end{align*} is the potential energy of the string. We refer to \cite[Section 2.6]{JDE17} for the details and for a discussion of links between the inextensible string \eqref{eq:geodgrav}, the geodesic interpretation of motion of an ideal incompressible fluid (associated with the names of Arnold and Brenier), and the optimal transport theory. In particular, Euler's equations for ideal incompressible fluid, together with the whip equations \eqref{eq:geod}, are examples of geodesic equations  on infinite-dimensional manifolds of volume preserving immersions endowed with the $L^2$ metric \cite{BMM16}. Apart from that,  \eqref{eq:geodgrav} can be recast \cite{JDE17} into a discontinuous system of conservation laws as well as into so-called ``total variation wave equation''.

Not much is known (see \cite{Preston-2011, JDE17} and the references therein) about existence and regularity of the exponential map on $\mathcal{A}$, which is equivalent to the question of well-posedness of the inextensible string equations \eqref{eq:geod} (or, more generally, \eqref{eq:geodgrav}) with the boundary conditions \eqref{e:bkg} and the initial conditions \begin{equation} \label{e:ikg} \eta(0,\cdot)=\eta_0\in\mathcal{A},\ \p_t\eta(0,\cdot)=v_0\in T_{\eta_0}\mathcal A.\end{equation} The applicability of standard geometric tools is impeded by low regularity of the Levi-Civita connection on $\mathcal{A}$, but PDE methods appeared to be slightly more helpful.
Smooth solutions to \eqref{eq:geod}, \eqref{e:bkg},  \eqref{e:ikg} exist locally in time \cite{Preston-2011}. Global solvability of \eqref{eq:geodgrav}, \eqref{e:bkg},  \eqref{e:ikg} has only been shown in the sense of generalized Young measures \cite{JDE17}.

Let us now derive a model for motion of a falling inextensible string (e.g., whip or flagellum) with one fixed and one free end, which is overdamped by a heavily dense environment. The motion of an inextensible string subject to a frictional force $f_d=c\eta_t$ (where $c$ is the damping coefficient) and a gravity force $f_g$ is governed by the system \begin{equation}\label{eq:geodoverd}
\begin{split}
\begin{cases}
\p_{tt}\eta(t,s) = \p_s(\varsigma(t,s) \p_s\eta(t,s))+f_g-f_d, \\
|\p_s\eta(t,s)| =1,
\end{cases}
\end{split}
\end{equation} cf. \eqref{eq:geodgrav}. Assume that the gravity is of the same order as the damping, that is, $f_g=c g$ for some constant vector $g$. Letting $\sigma=\varsigma/c$ and $c\to + \infty$ we formally deduce
\begin{equation}\label{eq:whip}
\begin{split}
\begin{cases}
\p_t\eta(t,s) &= \p_s(\sigma(t,s) \p_s\eta(t,s)) + g, \\
|\p_s\eta(t,s)| & =1.
\end{cases}
\end{split}
\end{equation}
We complement the system with the initial/boundary conditions \begin{align}\label{eq:bdry2}
\eta(t,1)=0, \quad \sigma(t,0)=0,\quad \eta(0,s)=\eta_0(s).
\end{align}
This system can be (formally) derived by considering the gradient flow on $\mathcal{A}$ driven by the potential energy $E(\eta)$
\begin{equation}
\p_t\eta = -\nabla_\mathcal{A} E(\eta),\label{e:grdflo}\end{equation} see Section \ref{s:gf}.  

Another way to feel the link between the conservative system \eqref{eq:geodgrav} and the gradient flow \eqref{eq:whip} is to employ the quadratic change of time \cite{BD18}. Indeed, introducing the ``new time'' $\theta:=t^2/2$ in \eqref{eq:geodgrav}, we derive \begin{equation}
\begin{split}
\begin{cases}
 \p_{\theta}\eta+ 2\theta \p_{\theta\theta}\eta = \p_s(\varsigma \p_s\eta)+g, \\
|\p_s\eta| =1,
\end{cases} 
\end{split} 
\end{equation} and for small  $\theta$  we approximately get  \eqref{eq:whip}. Thus, the study of \eqref{eq:whip} may shed more light on the geometry of the space of arcs since \eqref{eq:whip} might have many common features with its higher-order counterpart \eqref{eq:geodgrav}; in particular, the non-uniqueness of generalized solutions for both systems seems to be of similar nature, cf. Remark \ref{udwhip}. 

It is worth noting that the geometric ``uniformly compressing curve-shortening flow'', which we have introduced and studied in the companion paper \cite{SV19}, after a suitable change of variables becomes very similar to \eqref{eq:whip}: \begin{equation}\label{eq:mcf}
\begin{split}
\begin{cases}
\p_t\eta &= \p_s(\sigma \p_s\eta) + \eta, \\
|\p_s\eta| & =1,
\end{cases}
\end{split}
\end{equation} with the periodic boundary conditions. 

In what follows, we will always assume that the gravity is normalized: \begin{equation} \label{g=1} |g|=1.\end{equation} No generality is lost since it is possible to rescale $t$ and $\sigma$ in \eqref{eq:whip} to secure \eqref{g=1}.

In this paper we consider generalized solutions for the system \eqref{eq:whip}-\eqref{eq:bdry2} (cf. Definition \ref{defi:limit_sol}). Roughly speaking, a generalized solution pair $(\eta,\sigma)$ solves the equation $\p_t\eta=\p_s(\sigma\p_s\eta)+g$ almost everywhere and satisfies a relaxed constraint. Furthermore, any generalized solution which is $C^2$-regular solves the original system \eqref{eq:whip}-\eqref{eq:bdry2} (cf. Remark \ref{rmk:gene_sol}).  Our first result concerns the global existence of generalized solutions. More precisely, we prove
\begin{thm}
For any Lipschitz initial datum $\eta_0$ with $|\p_s\eta_0|=1$ and $\eta_0(1)=0$, there exists a solution $\eta:[0,\infty)\times [0,1]\rightarrow \R^d$, which solves \eqref{eq:whip}-\eqref{eq:bdry2} in the sense of Definition \ref{defi:limit_sol}. In addition, its Lagrange multiplier satisfies $\sigma\geq 0$ almost everywhere.
\end{thm}

The generalized solutions may not be unique.
For example, it is easy to check that $\eta_{-\infty}(s)=(s-1)g$ is a smooth stationary solution to \eqref{eq:whip} with $\sigma_{-\infty}(s)=-s\leq 0$, hence it is a generalized solution. However applying the above theorem with the initial datum $\eta_0=\eta_{-\infty}$ we obtain another generalized solution with $\sigma\geq 0$. The branching of trajectories is actually ubiquitous, and we refer to Section \ref{s:backward} for a more involved discussion on the non-uniqueness issue.

In the second part of the paper we study the long time asymptotics of a generalized solution. We show that any generalized solution with $\sigma \geq 0$ converges to the downwards vertical stationary solution with a universal exponential convergence rate.
\begin{thm}
Let $\eta$ be a generalized solution, whose Lagrange multiplier $\sigma$ is nonnegative almost everywhere. Let $\eta_\infty=(1-s)g$ be the downwards vertical stationary solution. Then there exists a universal constant $c_0>0$ such that
\begin{align*}
\left\|\eta(t,\cdot)-\eta_\infty(\cdot)\right\|^2_{L^2(0,1;\R^d)}\leq (2c_0)^{-1}\left(E(\eta(0,\cdot))-E(\eta_\infty)\right)e^{-c_0t}\text{ for any } t\in [0,\infty).
\end{align*}
\end{thm}

Let us briefly comment on the difficulties and main ideas used in the proofs. Firstly, the evolutionary system \eqref{eq:whip} is nonlinear due to the constraint $|\p_s\eta|=1$, and it is parabolic merely provided $\sigma> 0$.  The Lagrange multiplier $\sigma(t,\cdot)$ is formally determined by a nonlinear ODE involving $|\p_{ss}\eta(t,\cdot)|^2$ with Dirichlet-Neumann boundary condition (cf. Section \ref{s:gf}). Secondly, by exploring the geodesic convexity of the energy functional it is possible to show that the condition $\sigma\geq 0$ is preserved along the trajectories of our gradient flow, making the problem formally parabolic. However, the parabolicity is degenerate at the boundary (recall the boundary condition $\sigma(t,0)=0$) and possibly also in the interior. Thirdly, there are some intrinsic obstructions to regularity of solutions (cf. Section \ref{s:gf}).

Several related gradient flows of inextensible strings were studied in \cite{Koi96,Ok07,O08,OS10,Oel11,Oel14}. In those papers, additional fourth-order terms coming from the bending energy appear, which help to secure non-degenerate parabolicity of the equations and decrease the difficulties created by the Lagrange multiplier $\sigma$.

The induced Riemannian distance on $\mathcal{A}$ is not degenerate \cite{Preston-2012}. However, we do not expect  $\mathcal{A}$ to be a complete metric space since the Riemannian metric is weak, cf. \cite{BBM14}, and such issues are non-trivial even for strong metrics \cite{BV16}. Moreover, arbitrarily small pieces of geodesics might be non-minimizing \cite{Preston-2012}, so it is plausible that $\mathcal{A}$ is not a geodesic space (in the sense that any two points can be joined by a constant speed metric geodesic).  These peculiarities make the general theory of metric gradient flows \cite{AGS08} hardly applicable to wellposedness and long-time behaviour of our problem.

Our idea is to construct a suitable family of $L^2$-gradient flows on the ambient flat Hilbert space, which relax the constraint and at the same time approximate the original gradient flow on the manifold $\mathcal{A}$. Our approximation is totally different from the penalty method used in \cite{Oel14} for a related fourth-order flow. We show that the solutions to the approximation problems satisfy uniform energy and pointwise estimates (cf. Section \ref{s:approximation}). Furthermore, we are able to recover the constraint $|\p_s\eta|=1$ in a weak sense in the limit, and obtain a generalized solution to the original system. We point out that the Bakry-Emery strategy \cite{villani03topics,AGS08} is not applicable (cf. Section \ref{s:gf}) for the proof of exponential decay of the relative energy $E(\eta(t))-E(\eta_\infty)$, but we manage to prove the inequality between the relative energy and its dissipation in a different way.

The remainder of the paper is organized as follows: In Section \ref{s:gf} we explore the gradient flow structure of our problem, which provides a heuristic intuition for the rest of the argument. Having the gradient flow structure in mind, we consider in Section \ref{s:approximation} a family of approximation problems, which are quasilinear and parabolic. We show that solutions to the approximation problems satisfy uniform energy estimates (cf. Proposition \ref{prop:energy}) and $L^\infty$-estimates (cf. Proposition \ref{prop:max_uniform}). Moreover, we derive compactness results and in addition show that in the limit the constraint is satisfied in a weak sense (cf. Proposition \ref{prop:uniform}). In Section \ref{sec:limiting_prob} we define the generalized solution to the original problem (cf. Definition \ref{defi:limit_sol}) and prove that those limit solutions coming from the approximation problems are generalized solutions, which gives the main existence result (cf. Theorem \ref{thm:limit}). In Section \ref{s:exp} we show that any generalized solution converges exponentially fast to the downwards vertical stationary solution (cf. Theorem \ref{thm:exp}). Finally, in Section \ref{s:backward} we discuss backward solutions and the non-uniqueness issue.

\section{The gradient flow structure} \label{s:gf} In this section we make several formal heuristic observations related to the gradient flow structure of our problem.

\subsubsection*{The gradient of the potential energy}  It is clear that $-\nabla_{L^2} E(\eta)=g.$ Hence, one should have \begin{equation}\label{gradort}-\nabla_{\mathcal{A}} E(\eta)=P_\eta g,\end{equation} where $P_\eta g$ is the orthogonal projection of $g$ onto the tangent space $T_\eta \mathcal{A}$. However, $T_\eta \mathcal{A}$ is not closed in $L^2$ because of the condition at the fixed end in \eqref{e:compcons}, so the projection $P_\eta: H^2(0,1;\R^d)\rightarrow T_\eta\mathcal{A}$ fails to exist. This is related to non-existence of smooth solutions, and we will return to these issues by the end of Section \ref{s:gf}. Nevertheless, we still claim that if the ODE \eqref{eq:sigma}-\eqref{eq:sigma2}
\begin{gather} \label{eq:sigma}
\p_{ss}\sigma-|\p_{ss}\eta|^2 \sigma=0 \text{ for  } s\in (0,1),\\ \label{eq:sigma2} \sigma(0)=0,\ \p_{ss}\eta(1)\sigma(1)+ \p_s\eta(1)\p_s \sigma(1)=-g,
\end{gather}
is solvable for $\sigma$, then
 \begin{equation}\label{gradort1}
 P_\eta g:=g+\p_s(\sigma \p_s\eta)
 \end{equation}
(cf. a similar expression in \cite[Proposition 3.2]{Preston-2012}) fulfills the expected conditions for the image of $g$ under the orthogonal projection, namely, $P_\eta g\in T_\eta \mathcal{A}$ and $g-P_\eta g$ is $L^2$-orthogonal to any $v\in T_\eta \mathcal{A}$. Indeed, differentiating the constraint $|\p_s\eta|^2=1$ we find that $$\p_s\eta\cdot \p_{ss}\eta=0, \ \p_s\eta\cdot \p_{sss}\eta=-|\p_{ss}\eta|^2.$$ Hence, $$\p_s (P_\eta g)\cdot \p_s \eta=\p_{ss}\sigma-|\p_{ss}\eta|^2 \sigma=0.$$ Moreover, $(P_\eta g)(1)=0$ by \eqref{eq:sigma2}, so we have proved that $P_\eta g\in T_\eta \mathcal{A}$. Finally, for each $v\in T_\eta \mathcal{A}$, integration by parts implies $$\int_0^1(g-P_\eta g)\cdot v\, ds=\int_0^1\sigma \p_s\eta\cdot\p_s v\, ds=0.$$

\subsubsection*{Our problem as a gradient flow} Due to \eqref{gradort}, \eqref{gradort1}, any trajectory $\eta(t)$ of the gradient flow \eqref{e:grdflo} emanating from $\eta_0$ satisfies \eqref{eq:whip}, \eqref{eq:bdry2}. Conversely, any solution to \eqref{eq:whip}, \eqref{eq:bdry2} is a trajectory of the gradient flow. To see this, we just need  to show that \eqref{eq:whip}, \eqref{eq:bdry2} formally yield \eqref{eq:sigma}, \eqref{eq:sigma2}. As above, $|\p_s\eta|^2=1$ implies $\p_s\eta\cdot \p_{ss}\eta=0,\ \p_s\eta\cdot \p_{ts}\eta=0.$
Taking the scalar product of the first equation of \eqref{eq:whip} with $\p_s\eta$, we get \begin{equation}\p_t\eta \cdot \p_s
\eta=
 \p_s(\sigma\p_s\eta)\cdot \p_s\eta + g \cdot \p_s\eta= \p_s \sigma+ g \cdot \p_s\eta.\label{eq:7}\end{equation}
Differentiation in $s$ implies \begin{equation*}(\p_t\eta-g) \cdot \p_{ss}
\eta= \p_{ss} \sigma.\end{equation*} Remembering that $\p_t \eta-g=\p_s(\sigma \p_s\eta)$, we deduce  \eqref{eq:sigma}. Since $g+\p_s(\sigma \p_s\eta)(1)=\p_t \eta(1)=0$ by \eqref{eq:bdry2}, we retrieve \eqref{eq:sigma2}.

\subsubsection*{Estimating the Lagrange multiplier} By \eqref{eq:bdry2} and \eqref{eq:7}, $\sigma$ satisfies the Dirichlet-Neumann condition at the two ends
\begin{equation} \label{e:bksig}
\sigma(t,0)=0, \quad \p_s\sigma(t,1)=\cos \alpha ,
 \end{equation}
where $\alpha=\alpha(\eta(t))$ is the angle between the vertical direction $g$ and the tangent line to the arc at the fixed end: $$\cos \alpha=-g\cdot \p_s\eta(t,1).$$ By the maximum principle, for each fixed $t$, we find from \eqref{eq:sigma} that the function $\sigma(t,s)$ is either convex, increasing and non-negative (provided $\alpha$ is acute) or concave, decreasing and non-positive  (provided $\alpha$ is obtuse). Moreover, $\sigma(t,s)=0$ for all $s$ provided $\alpha$ is right. Consequently, $$|\partial_s \sigma(t,s)|\leq |\p_s\sigma(t,1)|=|\cos \alpha|\leq 1,$$ and thus $$|\sigma(t,s)|\leq |s\cos \alpha|\leq s.$$

\subsubsection*{Energy dissipation and the equilibria} We now compute the dissipation of the potential energy $E(t):=E(\eta(t,\cdot))$ along the trajectories of our gradient flow, employing integration by parts:
\begin{multline}\label{e:dissd} D(t):=-E'(t)=\int_0^1g \cdot \p_t\eta=\int_0^1g \cdot(g+\p_s(\sigma \p_s\eta))\\=1+\sigma(t,1)g\cdot \p_s\eta(t,1)=1-\sigma(t,1) \cos \alpha(\eta(t)).\end{multline}
The energy dissipation vanishes merely at the two equilibria of our problem, namely, at the downwards vertical state $(\eta_\infty,\sigma_\infty)=((1-s)g,s)$ (where $\alpha=0$ and the potential energy attains its minimum) and at the upward whip $(\eta_{-\infty},\sigma_{-\infty})=((s-1)g,-s)$ (where $\alpha=\pi$ and the potential energy attains its maximum). The fact that the potential energy achieves its extreme values at the corresponding equilibria will be rigorously  justified in Remark \ref{rmk:nonnegative}, although it is intuitively clear. Consequently, it makes sense to define the relative potential energy by $$\tilde{E}(t):=E(t)-E(\eta_\infty)\geq 0.$$

\subsubsection*{Geodesic convexity} We formally calculate the Hessian of the potential energy on $\mathcal{A}$ by the well-known formula $$\left\langle Hess_{\mathcal{A}}\, E (\eta_0) \cdot v_0,v_0\right\rangle_{T_{\eta_0}\mathcal{A}}=\frac {d^2}{dt^2}\Big|_{t=0} E(\eta(t)),$$
where $\eta(t)$ is the geodesic curve on $\mathcal{A}$ emanating from $\eta_0$ with the initial velocity field $v_0$, that is, a solution to \eqref{eq:geod}, \eqref{e:bkg},  \eqref{e:ikg}. We thus compute \begin{multline}\label{e:hess} \left\langle Hess_{\mathcal{A}}\, E (\eta_0) \cdot v_0,v_0\right\rangle_{T_{\eta_0}\mathcal{A}}=-\int_0^1g \cdot \p_{tt}\eta(0,s)\,ds=-\int_0^1g \cdot \p_s(\varsigma(0,s)\p_{s}\eta(0,s))\\=-\varsigma(0,1)g\cdot \p_{s}\eta_0(1)=\varsigma(0,1)\cos \alpha(\eta_0).\end{multline} But $\varsigma(s)=\varsigma(0,s)$ formally satisfies \begin{equation} \label{e:siggeod} \p_{ss}\varsigma-|\p_{ss}\eta_0|^2 \varsigma+|\p_{s}v_0|^2 =0,\
\varsigma(0)=0, \quad \p_s\varsigma(1)=0,
 \end{equation} which in particular yields $\varsigma\geq 0$, see \cite{Preston-2011,JDE17}. Thus the potential energy $E:\mathcal{A}\to \R$ is geodesically convex at a point $\eta_0$ if $\alpha(\eta_0)\leq \frac \pi 2$ (and concave if $\alpha(\eta_0)\geq \frac  \pi 2$). We however claim that it is possible to construct some sequences $\eta^\epsilon_0\in \mathcal{A}$,  $v_0^\epsilon\in {T_{\eta^\epsilon_0}\mathcal{A}}$, with $\alpha(\eta^\epsilon_0)=\alpha_0$ being a constant angle strictly between $0$ and $\pi$, and with $\|v_0^\epsilon\|_{L^2}$ bounded away from $0$, so that for the corresponding solutions to \eqref{e:siggeod}  one has $\varsigma^\epsilon(1)\to 0$.  Indeed, assume for definiteness that $d=3$ and $g=(0,0,-1)$. Then we can consider $$\eta^\epsilon_0(s)=(\epsilon\sin \alpha_0(\cos \frac s \epsilon -\cos \frac 1 \epsilon), \epsilon\sin \alpha_0(\sin \frac s \epsilon-\sin \frac 1 \epsilon ),(s-1)\cos \alpha_0)$$ and $$v^\epsilon_0(s)=(\epsilon\cos \alpha_0(\cos \frac s \epsilon -\cos \frac 1 \epsilon), \epsilon\cos \alpha_0(\sin \frac s \epsilon-\sin \frac 1 \epsilon ),(1-s)\sin \alpha_0).$$ Then \eqref{e:siggeod} becomes \begin{equation*}  \p_{ss}\varsigma^\epsilon- \frac {\sin^2 \alpha_ 0} {\epsilon^2}\varsigma^\epsilon+1 =0,\
 \varsigma^\epsilon(0)=0, \quad \p_s\varsigma^\epsilon(1)=0,
  \end{equation*} and an explicit computation shows that $\varsigma^\epsilon(1)\leq \frac {\epsilon^2}{\sin^2 \alpha_0} \to 0$. This counterexample prevents the inequality $$\left\langle Hess_{\mathcal{A}}\, E (\eta_0) \cdot v_0,v_0\right\rangle_{T_{\eta_0}\mathcal{A}}\geq \lambda \left\langle v_0,v_0\right\rangle_{T_{\eta_0}\mathcal{A}}, \ v_0\in T_{\eta_0}\mathcal{A}$$ to hold with a uniform $\lambda >0$ even if we restrict ourselves to small angles $\alpha(\eta_0)>0$. Moreover, $E$ is not uniformly $\lambda$-convex for any positive $\lambda$ even in a neighbourhood of $\eta_\infty$ defined by the inequality $\tilde E(\eta)<\varepsilon$ for any $\varepsilon$.

\subsubsection*{Exponential decay} A natural way to prove the exponential decay of the relative energy for a gradient flow \cite{villani03topics,villani08oldnew} is to establish a functional inequality between the relative energy $\tilde E$ and its dissipation $D$.  We are not able to apply the Bakry-Emery strategy \cite{villani03topics} based on the identity \begin{equation} \label{e:bemery}D'(t)=-2\left\langle Hess_{\mathcal{A}}\, E (\eta(t)) \cdot \nabla_\mathcal{A} E(\eta(t)),\nabla_\mathcal{A} E(\eta(t))\right\rangle_{T_{\eta(t)}\mathcal{A}}\end{equation} due to the lack of strict geodesic convexity. However, in Section \ref{s:exp} we will still manage to uniformly control the relative energy by its dissipation. The proof will heavily rely on Hardy type inequalities.

\subsubsection*{Non-negativity of $\sigma$ is conserved along the flow}  Setting $\tilde D(t) =D(t)$ for $\alpha(\eta(t))\leq \frac \pi 2$ and $\tilde D(t) =2-D(t)$ for $\alpha(\eta(t))\geq \frac \pi 2$, we derive from \eqref{e:dissd}, \eqref{e:hess} and \eqref{e:bemery} that $\tilde D$ decays along the trajectories of our gradient flow. If $\sigma(0,s)\geq 0$ for all $s$, or, equivalently, if $\alpha(\eta_0)\leq \frac \pi 2$, then $\tilde D(t)\leq 1$ for all $t>0$, hence $\alpha(\eta(t))$ remains acute (or at least right) along the whole trajectory, whence $\sigma(t,s)\geq 0$ for all $t$ and $s$. A similar property (conservation of non-negativity of $\varsigma$ along trajectories) was conjectured in \cite{JDE17} for the smooth solutions to the inextensible string equations \eqref{eq:geodgrav}, \eqref{e:bkg},  \eqref{e:ikg} but we do not know any proof yet. In Section \ref{sec:limiting_prob} we will construct generalized solutions with non-negative $\sigma$ for any initial datum. There is no paradox here since any initial datum $\eta_0$ with obtuse $\alpha$ can be uniformly approximated in $[0,1]$ by data with acute $\alpha$.

\subsubsection*{Mismatch at the fixed end}  We finish this section by observing that smooth solutions to our gradient flow can fail to exist. Indeed, let $\eta_0$ be any smooth function such that \begin{equation}|{\cos\alpha(\eta_0)}||\p_{ss}\eta_0(1)|< {1-|\cos\alpha(\eta_0)|}.\label{e:nonr}\end{equation} In particular, \eqref{e:nonr} holds for the whips with $|\p_{ss}\eta_0(1)|=0$ and $\alpha(\eta_0)\neq 0$ (different from the steady states), or for $\alpha=\frac \pi 2$. Then the corresponding regular solution would satisfy \begin{equation} \label{e:reg2}
0=\p_t\eta(0,1) = \p_s\sigma(0,1) \p_s\eta_0(1)+\sigma(0,1)\p_{ss}\eta_0(1) + g
\end{equation} But \begin{equation*}
|\p_s\sigma(0,1) \p_s\eta_0(1)+\sigma(0,1)\p_{ss}\eta_0(1)|\leq |\cos \alpha|(1+|\p_{ss}\eta_0(1)|)<1,
\end{equation*} which contradicts \eqref{e:reg2} since $|g|=1$.
Another way to see this discrepancy is to notice that the ODE problem \eqref{eq:sigma}, \eqref{eq:sigma2} is overdetermined. Hence, the orthogonal projection $P_{\eta_0} g$ may not exist, as anticipated at the beginning of Section \ref{s:gf}, although $P_{\eta(t)} g=\p_t \eta(t)$ must exist for the smooth trajectories $\eta(t)$ of the gradient flow, $t>0$.
All this might look like as a tame compatibility issue at the point $(t,s)=(0,1)$. However, any compatibility assumption would rule out an unacceptably large set of physically reasonable initial data, in particular, the open set in $C^2$ determined by \eqref{e:nonr}.

\section{Approximation}
\label{s:approximation}
In this section we approximate problem \eqref{eq:whip} by $L^2$-gradient flows in the ambient space. We derive uniform energy estimates for the approximation problem and show that the limiting functions are generalized solutions to the overdamped whip equations \eqref{eq:whip}.

First we rewrite the equation \eqref{eq:whip} as a first order system: let $\kappa:=\sigma \p_s\eta$, then $(\eta, \kappa, \sigma)\in \R^d\times \R^d\times \R$ solves
\begin{equation}\label{eq:whip2}
\begin{split}
\begin{cases}
\p_t\eta=\p_s\kappa+g\\
\kappa=\sigma\p_s\eta\\
\sigma=\kappa\cdot\p_s\eta,
\end{cases}
\end{split}
\end{equation}
and we keep the initial/boundary conditions \eqref{eq:bdry2}.
Noticing that $|\kappa|=|\sigma|$ by the constraint $|\p_s \eta|=1$, one can rewrite the second equation as $\kappa=sgn(\sigma)|\kappa|\p_s\eta$. If we assume furthermore that $\sigma\geq 0$ and $\kappa\neq 0$, then $\p_s\eta=\frac{\kappa}{|\kappa|}$.
This motivates us to consider the following
approximation problem: for 
\begin{equation}\label{eq:epsilon}
\epsilon\in (0,1/16)
\end{equation}
let
\begin{equation}\label{eq:F_eps}
F^\epsilon:\R^d\rightarrow \R^d,\quad F^\epsilon(\kappa):=\epsilon\kappa+\frac{\kappa}{\sqrt{\epsilon+|\kappa|^2}},
\end{equation}
and let
$$G^\epsilon(\tau):=(F^\epsilon)^{-1}(\tau).$$
The smallness condition \eqref{eq:epsilon} on $\epsilon$ is used in the later uniform estimates, for instance \eqref{e:rhside}. An explicit computation yields $\nabla G^\epsilon$ is positive definite and
$$\lambda_\epsilon (\tau)|\xi|^2\leq \nabla G^\epsilon(\tau)\xi\cdot \xi\leq \Lambda_\epsilon(\tau)|\xi|^2, \quad \forall \xi\in \R^d, \tau\in \R^d,$$
where $\lambda_\epsilon(\tau)$ and $\Lambda_\epsilon(\tau)$ satisfy
\begin{equation}\label{eq:G_eps}
\begin{split}
\lambda_\epsilon(\tau)=\frac{1}{\epsilon+(\epsilon+|G^\epsilon(\tau)|^2)^{-1/2}},\\
\Lambda_\epsilon(\tau)=\frac{\epsilon^{-1}}{1+(\epsilon+|G^\epsilon(\tau)|^2)^{-3/2}}.
\end{split}
\end{equation}
Then we consider the approximation problem:
\begin{equation}\label{eq:approx_whip}
\p_t\eta^\epsilon= \p_s(G^\epsilon(\p_s\eta^\epsilon))+g
\text{ in } (0,\infty)\times (0,1),
\end{equation}
with the initial/boundary conditions
\begin{equation}\label{eq:bdry}
\begin{split}
\eta^\epsilon(t,1)&=0, \quad \p_s\eta^\epsilon(t,0)=0,\\
\eta^\epsilon(0,s)&=\eta_0^\epsilon(s).
\end{split}
\end{equation}
Here $\eta^\epsilon_0$ are smooth functions in $(0,1)$, which are chosen to approximate the given initial datum $\eta_0$ in $C([0,1])$, see Remark \ref{rmk:initial} for the details. Here and in what follows we often omit the dimension $d$ for brevity. Since $\nabla G^\epsilon$ is smooth in its argument and positive definite, the above semi-linear system is well-posed: given any smooth initial datum $\eta^\epsilon_0$ satisfying \eqref{eq:bdry}, the existence of a unique smooth solution $\eta^\epsilon:C^{\infty}((0,T] \times [0,1]; \mathbb{R}^{d}) \cap C([0,T] \times [0,1]; \mathbb{R}^{d})$ to the above system follows from Amann's theory \cite{Amann} (cf. \cite[Proof of Theorem 4.2]{JDE17}). In comparison with the original equation \eqref{eq:whip2} we define
\begin{equation}\label{eq:kappa_sigma}
\begin{split}
\kappa^\epsilon := G^\epsilon(\p_s\eta^\epsilon),\quad
\sigma^\epsilon := G^\epsilon(\p_s\eta^\epsilon)\cdot \p_s\eta^\epsilon.
\end{split}
\end{equation}
In Section \ref{sec:limiting_prob} we will show that the limit of $(\eta^\epsilon, \kappa^\epsilon,\sigma^\epsilon)$ provides a solution to \eqref{eq:whip2}-\eqref{eq:bdry2}.

We may consider the associated energy to the above approximation system
\begin{equation}\label{eq:perturb_energy}
\mathcal{E}^\epsilon(\eta):=\int_0^1\widetilde{G}^\epsilon(\p_s\eta)\ ds+\int_0^1 (-g)\cdot \eta\ ds,
\end{equation}
where
\begin{equation}\label{eq:perturb_energy2}
\widetilde{G}^\epsilon:\R^d\rightarrow \R, \quad \widetilde{G}^\epsilon(u)=\epsilon \left(\frac{|G^\epsilon(u)|^2}{2}-\frac{1}{\sqrt{\epsilon+|G^\epsilon(u)|^2}}\right).
\end{equation}
Then \eqref{eq:approx_whip} can be interpreted as a gradient flow with respect to the flat Hilbertian structure inherited from $L^2$, which is driven by this functional:
\begin{align*}
\p_t\eta = -\nabla_{L^2} \mathcal E ^\epsilon(\eta).\end{align*}
We refer to the proof of Theorem \ref{thm:limit} for a detailed computation.

Now we derive some energy inequalities with uniform (in $\epsilon$) bounds for the solutions $\eta^\epsilon$ in terms of the initial datum.
We stress that in the sequel $C$ will always stand for a constant independent of $\epsilon$. For simplicity in the sequel we sometimes drop the dependence on $\epsilon$ and write $\eta=\eta^\epsilon$, $G=G^\epsilon$, etc. We also write $\Omega:=(0,1)$ and $Q_t:=(0,t)\times \Omega$ for $t\in (0,\infty]$.

\begin{prop}\label{prop:energy}
Given $\eta_0^\epsilon\in C^\infty(\Omega)$ as in Remark \ref{rmk:initial}, let $\eta^\epsilon$ be the solution to the approximation problem \eqref{eq:approx_whip}-\eqref{eq:bdry} in $Q_\infty$. Then for any $T\in (0,\infty)$
\begin{equation}\label{eq:energy11}
\max_{t\in [0,T]}\int_\Omega |\eta(t,\cdot)|^2\ ds +\int_{Q_T}|\p_s\eta|^2\ dsdt\leq 4e^T \left(1+\int_\Omega |\eta_0|^2 \ ds\right),
\end{equation}
\begin{equation}\label{eq:energy22}
\int_{Q_T}|\p_t\eta^\epsilon|^2+|\nabla G(\p_s\eta^\epsilon)\cdot \p_{ss}\eta^\epsilon|^2 \ dsdt \leq 3\mathcal{E}^\epsilon(\eta^\epsilon_0)+ 7e^{T/2}\left(\|\eta^\epsilon_0\|_{L^2(\Omega)}+1\right)+3\sqrt{\epsilon}.
\end{equation}
Here $\mathcal{E}^\epsilon$, as defined in \eqref{eq:perturb_energy}, is the associated energy for the approximation problem.
\end{prop}

\begin{proof}
\emph{Proof of \eqref{eq:energy11}.} We take the inner product of the equation with $\eta$ and integrate in space and time. After an integration by parts in space we obtain
\begin{align*}
\int_{Q_t}\p_t\eta \cdot \eta\ dsdt = -\int_{Q_t}G(\p_s\eta)\cdot\p_s\eta\ ds dt + \int_{Q_t} g\cdot\eta \ dsdt\text{ for any }t\in (0,T].
\end{align*}
Applying the Cauchy-Schwartz inequality we get
\begin{equation}\label{eq:energy1}
\begin{split}
&\quad \frac{1}{2}\int_\Omega |\eta(t,s)|^2\ ds +\int_{Q_t}G(\p_s\eta)\cdot\p_s\eta\ ds dt\\
& \leq \frac{1}{2}\int_\Omega |\eta_0(s)|^2\ ds+\frac{1}{2}\int_{Q_t}|\eta|^2 \ dsdt+ \frac{1}{2}\int_{Q_t}|g|^2\ dsdt.
\end{split}
\end{equation}
Observe that the second term on the left side of \eqref{eq:energy1} is nonnegative. By Gronwall's inequality
\begin{equation}\label{eq:eta_L2}
\int_\Omega |\eta(t,s)|^2\ ds \leq e^{t}\left(\int_{\Omega}|\eta_0|^2\ ds+1\right)\text{ for any }t\in [0,T].
\end{equation}
Maximizing both sides in $t$, we obtain the estimate for $\max_{t\in [0,T]}\|\eta(t,\cdot)\|_{L^2(\Omega)}$.

To show the estimate for $\|\p_s\eta\|_{L^2(Q_T)}$, we assume for now the following estimate holds (whose proof will be provided later):
\begin{equation}\label{eq:partial_eta}
G(\p_s\eta)\cdot\p_s\eta\geq \frac{1}{\epsilon+\sqrt{\epsilon}}|\p_s\eta|^2 \text{ in } \{(t,s)\in Q_t: |\p_s\eta(t,s)|\geq 1+\sqrt{\epsilon}\}.
\end{equation}
Applying \eqref{eq:partial_eta} to \eqref{eq:energy1} we obtain
\begin{equation*}
\begin{split}
\int_{Q_t}|\p_s\eta|^2\ dsdt&\leq ({\epsilon+\sqrt{\epsilon}})\int_{Q_t}G(\p_s\eta)\cdot\p_s\eta\ dsdt + (1+\sqrt{\epsilon})^2|Q_t|\\
&\leq \frac{{\epsilon+\sqrt{\epsilon}}}{2}\int_\Omega |\eta_0(s)|^2ds+\frac{{\epsilon+\sqrt{\epsilon}}}{2}\int_{Q_t}|\eta|^2 \ dsdt+ \left((1+\sqrt{\epsilon})^2+\frac{{\epsilon+\sqrt{\epsilon}}}{2}\right)|Q_t|.
\end{split}
\end{equation*}
This together with \eqref{eq:eta_L2} (after taking the suprema over $t$) yields
\begin{equation*}
\max_{t\in [0,T]}\int_\Omega |\eta(t,\cdot)|^2\ ds +\int_{Q_T}|\p_s\eta|^2\ dsdt\leq 4e^T \left(1+\int_\Omega |\eta_0|^2 \ ds\right).
\end{equation*}

It remains to prove \eqref{eq:partial_eta}.
By the explicit expression of $F$ in \eqref{eq:F_eps} and $G=F^{-1}$ one has  $\tau=\epsilon G(\tau)+(\epsilon+|G(\tau)|^2)^{-1/2}G(\tau)$, $\tau\in \R^d$. Thus $G(\tau)\cdot \tau =\frac{|\tau|^2}{\epsilon+(\epsilon+|G(\tau)|^2)^{-1/2}}$. Using the monotonicity of $r\mapsto \tilde{F}(r):=\epsilon r+\frac{r}{\sqrt{\epsilon+r^2}}$ in $[0,\infty)$ and $\tilde{F}(\frac{1}{\sqrt{\epsilon}})<1+\sqrt{\epsilon}$ one can conclude that \begin{equation}\label{eq:estimate_G}
|G(\tau)|\geq \frac{1}{\sqrt{\epsilon}}\text{ in }\{\tau: |\tau|\geq 1+\sqrt{\epsilon}\}.
\end{equation}
 Hence $$G(\tau)\cdot \tau \geq \frac{|\tau|^2}{\epsilon+(\epsilon+\epsilon^{-1})^{-1/2}}\geq \frac{1}{{\epsilon+\sqrt{\epsilon}}}|\tau|^2$$ if $|\tau|\geq 1+\sqrt{\epsilon}$. This implies the claimed estimate \eqref{eq:partial_eta}.\\

\emph{Proof of \eqref{eq:energy22}.} We take the inner product of the equation with $\p_t \eta$ and integrate over $Q_t$. After some integration by parts we obtain
\begin{align*}
\int_{Q_t}|\p_t \eta|^2\ dsdt = -\int_{Q_t} G(\p_s\eta)\p_{st}\eta \ ds dt + \int_\Omega g\cdot \eta(t,\cdot)\ ds-\int_\Omega g\cdot \eta_0 \ ds.
\end{align*}
Here we have used the boundary conditions $\p_t\eta(t,1)=0$ and $\p_s\eta(t,0)=0$. Since $ G(\p_s\eta)\p_{st}\eta= \p_t \widetilde{G}(\p_s\eta)$ (cf. \eqref{eq:G} for detailed computation, and recall the definition of $\widetilde{G}$ in \eqref{eq:perturb_energy2}), we have
\begin{align*}
&\int_{Q_t}|\p_t \eta|^2\ dsdt+ \int_\Omega \widetilde{G}(\p_s\eta)(t,\cdot)\ ds+\int_\Omega (-g)\cdot \eta(t,\cdot)\ ds\\
& = \int_\Omega \widetilde{G}(\p_s\eta_0)\ ds +  \int_\Omega (-g)\cdot \eta_0 \ ds,
\end{align*}
or in other words $$\int_{Q_t}|\p_t \eta|^2\ dsdt+ \mathcal{E}(\eta(t,\cdot))\leq \mathcal{E}(\eta_0).$$ This in particular gives the decay of the energy
$$\mathcal{E}(\eta(t,\cdot))\leq \mathcal{E}(\eta_0)<\infty\text{ for any } t\in (0,T].$$
By virtue of $\widetilde{G}\geq -\sqrt{\epsilon}$, \eqref{eq:eta_L2} and recalling the definition of $\mathcal{E}$ in \eqref{eq:perturb_energy}, we have
$$\mathcal{E}(\eta(t,\cdot))\geq -e^{T/2}(\|\eta_0\|_{L^2(\Omega)}+1) -\sqrt{\epsilon}.$$
Thus
\begin{equation*}
\int_{Q_T}|\p_t \eta|^2\ dsdt\leq \mathcal{E}(\eta_0)+ e^{T/2}(\|\eta_0\|_{L^2(\Omega)}+1)+\sqrt{\epsilon}.
\end{equation*}
On the other hand, from the equation $\p_s G(\p_s\eta)=\p_t\eta -g$ we deduce
\begin{equation*}
\int_{Q_T}|\nabla G\cdot \p_{ss}\eta|^2 \ ds dt=\int_{Q_T}|\p_sG(\p_s\eta)|^2 \ dsdt \leq 2 \int_{Q_T}|\p_t\eta|^2+ 2\int_{Q_T}|g|^2.
\end{equation*}
The last two inequalities together yield \eqref{eq:energy22}.
\end{proof}

\begin{rmk}\label{rmk:initial}
Given any initial datum $\eta_0\in W^{1,\infty}([0,1])$ with $\eta_0(1)=0$ and $|\p_s\eta_0(s)|\leq 1$ to the original problem \eqref{eq:whip}-\eqref{eq:bdry2}, it is possible to find smooth $\eta_0^\epsilon$ satisfying $\eta^\epsilon_0(1)=0$ and $\p_s\eta^\epsilon_0(0)=0$, such that $\eta_0^\epsilon\rightarrow \eta_0$ uniformly and $|\p_s\eta_0^\epsilon|\leq 1$ (for example, one can first extend $\eta_0$ to $\R$, then consider $\widetilde{\eta}_0^\epsilon = \eta_0\ast \phi_\epsilon$, where $\phi_\epsilon = \phi(s/\epsilon)/\epsilon$, $\phi\geq 0$, $\|\phi\|_{L^1}=1$ is the standard mollifier. It is straightforward to check that $\widetilde{\eta}^\epsilon_0\rightarrow \eta_0$ in $C([0,1])$ and $|\p_s\widetilde{\eta}^\epsilon_0|\leq 1$. To make sure that the boundary conditions are satisfied, one applies an additional cut-off to $\p_s\widetilde{\eta}^\epsilon_0$ around $0$, i.e. let $\eta^\epsilon_0(s):=-\int_s^1\chi_\epsilon (x)\p_s\widetilde{\eta}^\epsilon_0(x)dx $, where $\chi_\epsilon\in C^\infty(\R)$ satisfies $0\leq \chi_\epsilon\leq 1$,  $\chi_\epsilon =1$ in $[\epsilon,\infty)$ and $\chi_\epsilon=0$ in $(-\infty,0)$).

For such approximations the initial energies $\mathcal{E}^\epsilon(\eta_0^\epsilon)$ are uniformly bounded in $\epsilon\in (0,1/16)$. To see this, let $\tilde{F}$ be the same as in the proof for \eqref{eq:estimate_G}. Observing that $\tilde{F}(\frac{1}{\sqrt{\epsilon}})>1$ and using the monotonicity of $\tilde{F}$ one has $|G^\epsilon(\tau)|\leq \frac{1}{\sqrt{\epsilon}}$ if $|\tau|\leq 1$. Then it is not hard to see from \eqref{eq:perturb_energy} that $\mathcal{E}^\epsilon(\eta_0^\epsilon)\leq 2$ if $|\p_s\eta_0^\epsilon|\leq 1$.
\end{rmk}

Next we estimate $\sup_{Q_T}|\p_s\eta^\epsilon|$. The proof is inspired by the $L^\infty$-estimate for the parabolic quasi-linear system (cf. \cite[Chapter VII]{LSU}). For this we take $\p_s$ of the equation and let $u^\epsilon:=\p_s\eta^\epsilon$. Then $u^\epsilon$ solves the problem
\begin{equation*}
\begin{split}
\begin{cases}
\p_t u^\epsilon = \p_s(\nabla G(u^\epsilon)\p_su^\epsilon),\\
u^\epsilon(t,0)=0,\quad \p_sG(u^\epsilon)(t,1)=-g,\\
u^\epsilon(0,s)=\p_s \eta^\epsilon_0(s)= :u^\epsilon_0(s).
\end{cases}
\end{split}
\end{equation*}
We aim to prove the weak type maximum principle:

\begin{prop}\label{prop:max_uniform}
Let $u^\epsilon$ and $u^\epsilon_0$ be as above. Then there exists a positive constant $C=C(T)$ such that
$$\sup_{Q_T}|u^\epsilon|\leq C\sup _\Omega |u^\epsilon_0|.$$
\end{prop}

\begin{proof}
In the proof we drop the dependence on $\epsilon$ for simplicity. Let $k$ be a constant and $k\geq \max\{\sup _\Omega |u_0|, (1+\sqrt{\epsilon})^2\}$. We take the inner product of the equation with $u(|u|^2-k)_+$ and denote $v^{(k)}:=(|u|^2-k)_+$. Integrating over $Q_T$ and after an integration by part we obtain
\begin{equation} \label{e:rhside} \begin{split}
\int_{Q_T} \p_t u \cdot u v^{(k)} \ dsdt = \int_0^T (-g)\cdot ( uv^{(k)})(t,1) \ dt \\
-\int_{Q_T}\nabla G(u)\p_s u\cdot \p_s u v^{(k)}\ dsdt- \int_{Q_T} \nabla G(u)\p_su\cdot u\p_s v^{(k)}\ dsdt.  \end{split}
\end{equation}
We denote $A_k(t):=\{s\in \Omega: |u|^2(t,s)>k\}$ and $Q_k:=\{(t,s)\in Q_T: |u|^2(t,s)>k\}$. First we note that using $$\p_t u \cdot u v^{(k)} =\frac{1}{4} \p_t |v^{(k)}|^2$$ and the choice of $k$ (such that $v^{(k)}(0,\cdot)=0$) we deduce
\begin{align*}
\int_{Q_T} \p_t u \cdot u v^{(k)} \ dsdt = \frac{1}{4}\int_\Omega |v^{(k)}|^2(T,\cdot) \ ds .
\end{align*}
For the last two terms in the right hand side of \eqref{e:rhside}, using $\lambda_\epsilon(\tau)\geq 1$ if $|\tau|^2\geq (1+\sqrt{\epsilon})^2$, which follows from \eqref{eq:epsilon} and from the expression for $\lambda_\epsilon$ in \eqref{eq:G_eps} and \eqref{eq:estimate_G}, we get
\begin{align*}
\int_{Q_T}\nabla G(u)\p_s u\cdot \p_s u v^{(k)}\ dsdt \geq \int_{Q_T} |\p_s u|^2 v^{(k)} \ dsdt.
\end{align*}
Similarly,
\begin{align*}
 \int_{Q_T} \nabla G(u)\p_su\cdot u \p_s v^{(k)}\ dsdt \geq \frac{1}{2} \int_{Q_T}|\p_s v^{(k)}|^2 \ dsdt.
\end{align*}
For the first term on the right hand side of \eqref{e:rhside}, we use the the boundary condition $u(t,0)=0$ and fundamental theorem of calculus to write
\begin{align*}
\int_0^T (-g)\cdot ( uv^{(k)})(t,1) \ dt & = \int_{Q_T} (-g)\cdot \p_s(u v^{(k)}) \ dsdt\\
&= \int_{Q_T} (-g)\cdot \p_s u v^{(k)} \ dsdt + \int_{Q_T} (-g)\cdot u \p_s v^{(k)} \ dsdt.
\end{align*}
By the Cauchy-Schwartz inequality and that $\supp (v^{(k)}), \supp (\p_sv^{(k)})\subset Q_k$,
\begin{align*}
\int_0^T (-g)\cdot ( uv^{(k)})(t,1) \ dt& \leq \frac{1}{4}\int_{Q_k} |\p_su|^2 v^{(k)} \ dsdt + \int_{Q_k}v^{(k)} \ dsdt \\
&+\frac{1}{4}\int_{Q_k}|\p_s v^{(k)}|^2 \ dsdt +  \int_{Q_k} |u|^2 \ dsdt.
\end{align*}
Combining all above inequalities we deduce that there exists  a universal constant $C$ (independent of $\epsilon$) such that
\begin{align*}
&\sup_{t\in[0,T]} \int_\Omega |v^{(k)}|^2(t,\cdot) \ ds + \int_{Q_T} |\p_s u|^2 v^{(k)} \ dsdt +  \int_{Q_T}|\p_s v^{(k)}|^2 \ dsdt\\
&\leq C\int_{Q_k} v^{(k)} \ dsdt + C \int_{Q_k} |u|^2 \ dsdt\\
&\leq C \int_{Q_k} v^{(k)} \ dsdt + C\int_{Q_k} k \ dsdt,
\end{align*}
where in the last inequality we have used $|u|^2 \leq (|u|^2-k)_++ k= v^{(k)}+k$. Using the Cauchy-Schwartz and H\"older inequalities, we have
\begin{align*}
(\text{RHS})\leq \frac{1}{4}\sup_{t\in [0,T]}\int_\Omega |v^{(k)}|^2(t,\cdot) \ ds  + C \int_0^T |A_k(t)|\ dt + Ck\int_0^T |A_k(t)| \ dt.
\end{align*}
Note that the first term can be absorbed by the left hand side. Thus,
\begin{align*}
\sup_{t\in[0,T]} \int_\Omega |v^{(k)}|^2(t,\cdot) \ ds +\int_{Q_T}|\p_s v^{(k)}|^2 \ dsdt \leq Ck \int_0^T |A_k(t)| \ dt.
\end{align*}
By Theorem 6.1 in Chapter II of \cite{LSU} we conclude
\begin{equation*}
\sup_{Q_T}|u|^2(t,s) \leq 2 (1+C) \sup_{\Omega}|u_0|^2(s).
\end{equation*}
This completes the proof.
\end{proof}

In the next proposition we derive compactness results and show that weak limits of approximation solutions solve almost everywhere the original evolutionary equation with a slightly relaxed constraint.

\begin{prop}\label{prop:uniform}
Given any $T>0$, let $\eta^\epsilon$ be a solution to \eqref{eq:approx_whip} in $Q_T$ with the initial/boundary conditions \eqref{eq:bdry}. Let $(\kappa^\epsilon,\sigma^\epsilon)$ be as in \eqref{eq:kappa_sigma}. Assume that $\|\eta_0^\epsilon\|_{L^2(\Omega)}$ and $\mathcal{E}^\epsilon(\eta_0^\epsilon)$ are bounded uniformly in $\epsilon$. Then
\begin{itemize}
\item [(i)] Along a subsequence $\epsilon\rightarrow 0$ one has
\begin{align*}
\eta^\epsilon &\rightarrow \eta \text{ weakly* in } L^\infty(0,T; W^{1,\infty}(\Omega))^d \text{ and strongly in } L^2(Q_T)^d,\\
\p_t\eta^\epsilon &\rightarrow \p_t \eta \text{ weakly in } L^2(Q_T)^d,\\
\sigma^\epsilon &\rightarrow \sigma \text{ weakly in } L^2(0,T; H^1(\Omega)).
\end{align*}
\item [(ii)] The limit $(\eta,\sigma)$ satisfies $$\sigma \p_s\eta\in L^2(0,T; H^1(\Omega))^d$$ and solves \eqref{eq:whip}-\eqref{eq:bdry2} in the sense that
\begin{align*}
& \p_t\eta = \p_s(\sigma\p_s\eta) + g, \qquad\sigma \left(|\p_s\eta|^2 -1\right ) =0 \text{ a.e. in } Q_T,\\
& \eta(t,1)=0\text{ for all } t,\  \eta(0,s)=\eta_0(s)  \text{ for all } s,\ \sigma(t,0)=0  \text{ for a.e. } t.
\end{align*}
\end{itemize}
\end{prop}

\begin{rmk} \label{rmk:ib}
The initial and boundary conditions make sense due to the embeddings \eqref{e:emb} and \eqref{e:emb2} which appear below in the proof.
\end{rmk}

\begin{proof}[Proof of Proposition \ref{prop:uniform} (i).]
The compactness results for $\eta^\epsilon$ follow immediately from the uniform energy bound in Proposition \ref{prop:energy} and the $L^\infty$ bound for $\p_s\eta^\epsilon$ from Proposition \ref{prop:max_uniform}. We will mainly derive the uniform boundedness of $\sigma^\epsilon$.

By a direct computation $\p_s\sigma^\epsilon = \p_s \kappa^\epsilon\cdot \p_s\eta^\epsilon + \kappa^\epsilon\cdot \p_{ss}\eta^\epsilon$. We will estimate the two terms in the summation separately. First by \eqref{eq:energy22} and Poincar\'e's inequality we immediately obtain that
$\kappa^\epsilon=G^\epsilon(\p_s\eta^\epsilon)$ are uniformly bounded in $L^2(0,T; H^1(\Omega))$.
This together with Proposition~\ref{prop:max_uniform} gives the uniform boundedness of $\p_s \kappa^\epsilon\cdot \p_s\eta^\epsilon$ in $L^2(Q_T)$. To estimate $\kappa^\epsilon\cdot \p_{ss}\eta^\epsilon$, we are going to show that
\begin{align}\label{eq:inequ}
|\nabla G^\epsilon(\p_s\eta^\epsilon)\p_{ss}\eta^\epsilon|\geq (C(T)+2)^{-1}|G^\epsilon(\p_s\eta^\epsilon)\cdot \p_{ss}\eta^\epsilon|.
\end{align}
Here $C(T)$ is the constant in Proposition \ref{prop:max_uniform}. To see this we first observe from the the explicit expression of $\lambda_\epsilon$ in \eqref{eq:G_eps} that
\begin{align*}
\lambda_\epsilon(\tau)=\frac{\sqrt{\epsilon+|G^\epsilon(\tau)|^2}}{\epsilon\sqrt{\epsilon+|G^\epsilon(\tau)|^2}+1}\geq \frac{|G^\epsilon(\tau)|}{\epsilon |G^\epsilon(\tau)|+1}, \ \tau\in \R^d.
\end{align*}
Thus
\begin{align*}
|\nabla G^\epsilon(\p_s\eta^\epsilon)\p_{ss}\eta^\epsilon|\geq |\lambda_\epsilon \p_{ss}\eta^\epsilon|\geq \frac{\left| G^\epsilon(\p_s\eta^\epsilon)\right|}{\epsilon | G^\epsilon(\p_s\eta^\epsilon)|+1}\left|\p_{ss}\eta^\epsilon \right|
\end{align*}
By the expression of $F^\epsilon$ and Proposition \ref{prop:max_uniform},
$$\epsilon|G^\epsilon(\p_s\eta^\epsilon)|=\epsilon|\kappa^\epsilon|\leq |\p_s\eta^\epsilon|+1\leq C(T)+1.$$ Combining the above two inequalities we thus obtain \eqref{eq:inequ}.

The estimate \eqref{eq:inequ} immediately implies that $\kappa^\epsilon\cdot \p_{ss}\eta^\epsilon=G^\epsilon(\p_s\eta^\epsilon)\cdot \p_{ss}\eta^\epsilon$ are uniformly bounded in $L^2(Q_T)$ (since $\nabla G^\epsilon(\p_s\eta^\epsilon)\p_{ss}\eta^\epsilon$ are uniformly bounded in $L^2(Q_T)$ by \eqref{eq:energy22}). Thus we have shown the uniform boundedness of $\sigma^\epsilon$ in $L^2(0,T; H^1(\Omega))$.

\emph{Proof of (ii).} From the uniform energy bound in Proposition \ref{prop:energy} we see that there exists $\kappa:=\lim\kappa^\epsilon$ in the weak topology of   $L^2(0,T; H^1(\Omega))$. Let us show that \begin{equation}\label{e:kseta}\begin{split}\kappa=\sigma \p_s\eta,\\ \sigma=\kappa\cdot \p_s\eta\end{split}\end{equation} a.e. in $Q_T$. Since both sides of the equalities \eqref{e:kseta} are integrable on $Q_T$, it suffices to prove \eqref{e:kseta} in the sense of distributions, i.e. that for any $\varphi\in L^2(0,T;H^1(\Omega))$
\begin{equation}\label{eq:kappa_equ}
\int_{Q_T} \kappa \varphi \ dsdt= -\int_{Q_T} \sigma\eta \p_s \varphi \ dsdt-\int_{Q_T} \p_s\sigma \eta\varphi \ dsdt,
\end{equation}
\begin{equation}\label{eq:sig_equ}
\int_{Q_T} \sigma \varphi \ dsdt= -\int_{Q_T} \kappa\cdot\eta\p_s \varphi \ dsdt-\int_{Q_T} \p_s\kappa \cdot\eta \varphi \ dsdt.
\end{equation}
Indeed, integrating by parts the distributional form of the equality $\sigma^\epsilon=\kappa^\epsilon\cdot \p_s\eta^\epsilon$ we see that
\begin{align*}
\int_{Q_T} \sigma^\epsilon  \varphi \ dsdt=-\int_{Q_T} \kappa^\epsilon\cdot\eta^\epsilon \p_s \varphi \ dsdt -\int_{Q_T} \p_s\kappa^\epsilon\cdot \eta^\epsilon\varphi \ dsdt,
\end{align*} and due to the strong compactness property of $\{\eta^\epsilon\}$ in (i) we can pass to the limit to get \eqref{eq:sig_equ}.
We now claim
\begin{align}\label{eq:sigma_equ}
\lim_{\epsilon\rightarrow 0}\int_{Q_T} |\kappa^\epsilon|\left||\p_s\eta^\epsilon|^2-1\right| =0.
\end{align}
Before proving the claim we show how \eqref{eq:kappa_equ} follows from \eqref{eq:sigma_equ}. Indeed, with \eqref{eq:sigma_equ} at hand and noting that $\kappa^\epsilon|\p_s\eta^\epsilon|^2= (\kappa^\epsilon\cdot \p_s\eta^\epsilon) \p_s\eta^\epsilon= \sigma^\epsilon \p_s\eta^\epsilon$ we have
\begin{equation}\label{eq:L1}
\lim_{\epsilon\rightarrow 0}\|\sigma^\epsilon\p_s\eta^\epsilon-\kappa^\epsilon\|_{L^1(Q_T)}=0.
\end{equation}
In particular, for any $\varphi\in L^2(0,T;H^1(\Omega))$
\begin{align*}
\lim_{\epsilon\rightarrow 0}\int_{Q_T} \kappa^\epsilon \cdot \varphi \ dsdt = \lim_{\epsilon\rightarrow 0} \int_{Q_T} \sigma^\epsilon\p_s\eta^\epsilon \cdot \varphi \ dsdt.
\end{align*}
An integration by parts applied to the integral on the right side gives
\begin{align*}
\lim_{\epsilon\rightarrow 0}\int_{Q_T} \kappa^\epsilon \cdot \varphi \ dsdt=\lim_{\epsilon\rightarrow 0}\left(-\int_{Q_T} \sigma^\epsilon\eta^\epsilon \p_s \varphi \ dsdt -\int_{Q_T} \p_s\sigma^\epsilon \eta^\epsilon \varphi \ dsdt\right).
\end{align*}
This together with the compactness results in (i) yields \eqref{eq:kappa_equ}.

We now provide the proof for \eqref{eq:sigma_equ}. By the definition of $F^\epsilon$ in \eqref{eq:F_eps},
\begin{align*}
|\p_s\eta^\epsilon|-1=|F^\epsilon(\kappa^\epsilon)|-1&=\epsilon|\kappa^\epsilon|+\frac{|\kappa^\epsilon|}{\sqrt{\epsilon+|\kappa^\epsilon|^2}}-1\\
&=\epsilon |\kappa^\epsilon|-\frac{\epsilon}{\sqrt{\epsilon+|\kappa^\epsilon|^2}\left(\sqrt{\epsilon+|\kappa^\epsilon|^2}+|\kappa^\epsilon|\right)}.
\end{align*}
Thus
\begin{equation*}
\begin{split}
|\kappa^\epsilon|\left||\p_s\eta^\epsilon|-1\right|
&=\epsilon\left||\kappa^\epsilon|^2-\frac{|\kappa^\epsilon|}{\sqrt{\epsilon+|\kappa^\epsilon|^2}\left(\sqrt{\epsilon+|\kappa^\epsilon|^2}+|\kappa^\epsilon|\right)}\right|\\
&\leq \epsilon|\kappa^\epsilon|^2 +\frac{\epsilon}{\sqrt{\epsilon+|\kappa^\epsilon|^2}} \leq  \epsilon|\kappa^\epsilon|^2+\sqrt{\epsilon}.
\end{split}
\end{equation*}
This together with Proposition~\ref{prop:max_uniform} yields
\begin{align*}
\int_{Q_T}|\kappa^\epsilon|\left||\p_s\eta^\epsilon|^2-1\right|&\leq (C(T)+1)\int_{Q_T}|\kappa^\epsilon|\left||\p_s\eta^\epsilon|-1\right|\\
&\leq (C(T)+1)\int_{Q_T}\epsilon|\kappa^\epsilon|^2+(C(T)+1)\sqrt{\epsilon}|Q_T|.
\end{align*}
Since $\kappa^\epsilon$ are uniformly bounded in $L^2(Q_T)$ by (i), passing to the limit $\epsilon\rightarrow 0$ we recover \eqref{eq:sigma_equ}.

Passing to the weak limit in $L^2(Q_T)$ in the equation $\p_t\eta^\epsilon=\p_s\kappa^\epsilon+g$ and using \eqref{e:kseta} we obtain $\p_t\eta = \p_s(\sigma\p_s\eta) + g$. To get $\sigma(|\p_s\eta|^2-1)=0$, it suffices to express $\kappa$ from the first equality in \eqref{e:kseta} and to substitute the result into the second one.

We now observe that by the Aubin-Lions-Simon theorem, \begin{equation}  \label{e:emb} L^\infty (0,T; W^{1,\infty}(\Omega))\cap H^1(0,T; L^2(\Omega)) \subset  C([0,T]; C(\overline\Omega)),\end{equation} and the embedding is compact. Without loss of generality, we may therefore assume that $\eta^\epsilon\to \eta$ strongly in $C([0,T]\times \overline\Omega)$. Hence, $\eta_0^\epsilon=\eta^\epsilon(0,\cdot)\to \eta(0,\cdot)$ uniformly in $s$, thus $\eta(0,\cdot)=\eta_0$. In a very similar way we obtain the required boundary condition at the fixed end.

To check the validity of the boundary condition for $\sigma$, we make the following trick. We swap the variables $t$ and $s$, noting that $\sigma^\epsilon$ are uniformly bounded and weakly converging in $H^1(0,1; L^2(0,T))$. Employing, for instance, \cite[Corollary 2.2.1]{ZV08}, we get \begin{equation}  \label{e:emb2} H^1(0,1; L^2(0,T))\subset C([0,1]; L^2(0,T)).\end{equation} Hence, by the Aubin-Lions-Simon theorem, the embedding \begin{equation*}   H^1(0,1; L^2(0,T))\subset C([0,1]; H^{-1}(0,T))\end{equation*} is compact, whence we may assume that $\sigma^\epsilon\to \sigma$ strongly in $C([0,1]; H^{-1}(0,T))$. Using \eqref{eq:bdry} and \eqref{eq:kappa_sigma}, we get $0=\sigma^\epsilon(\cdot,0)\to \sigma(\cdot,0)$ in $H^{-1}(0,T)$.  Consequently, $\sigma(t,0)=0$ in $L^2(0,T)$ and for a.e. $t$.
\end{proof}

\begin{prop}\label{prop:lim_eta}
Let $(\eta, \sigma)$ be a limiting solution obtained in Proposition \ref{prop:uniform}. Then
\begin{itemize}
\item[(i)] $|\p_s\eta(t,s)|\leq 1$ for a.e. $(t,s)\in Q_T$;
\item[(ii)] $\sigma(t,s)\geq 0$ for a.e. $(t,s)\in Q_T$.
\end{itemize}
\end{prop}
\begin{proof}
(i) Observe that the set $$B=\{\xi \in L^2(Q_T):|\xi(t,s)|\leq 1\ \textrm{for}\ \textrm{a.e.}\ (t,s)\in Q_T\}$$ is weakly closed in $L^2(Q_T)$. By the definition of $F^\epsilon$ in \eqref{eq:F_eps} we have that $$\epsilon\kappa^\epsilon+\frac{\kappa^\epsilon}{\sqrt{\epsilon+|\kappa^\epsilon|^2}}=\p_s\eta^\epsilon.$$ As $\epsilon\to 0$, the first term in the left-hand side goes to zero in $L^2(Q_T)$. Since the second term on the left belongs to $B$, so does the weak limit  $\p_s\eta$ of the right-hand side.

(ii) Recall the definition of $\sigma^\epsilon$ in \eqref{eq:kappa_sigma}: $\sigma^\epsilon=\kappa^\epsilon\cdot \p_s\eta^\epsilon=\kappa^\epsilon\cdot F^\epsilon(\kappa^\epsilon)$ for each $\epsilon>0$. Thus from the definition of $F^\epsilon$ we obtain $\sigma^\epsilon\geq 0$. This implies that in the limit $\sigma\geq 0$.
\end{proof}

\section{Limiting problem}\label{sec:limiting_prob}
\begin{defi}\label{defi:limit_sol}
Let $\Omega:=(0,1)$. Given an initial datum $\eta_0\in W^{1,\infty}(\Omega)^d$ with $\eta_0(1)=0$ and $|\p_s\eta_0(s)|\leq 1$ for almost every $s\in \Omega$, we call a pair $(\eta,\sigma)$ a generalized solution to \eqref{eq:whip}, \eqref{eq:bdry2} in $Q_\infty:=(0,\infty)\times \Omega$ if
\begin{itemize}
\item[(i)]  $\eta\in L^\infty_{loc}([0,\infty); W^{1,\infty}(\Omega))^d$, $\p_t\eta\in L^2_{loc}([0,\infty);L^2(\Omega))^d$, $\sigma\in L^2_{loc}([0,\infty); H^1(\Omega))$ and $\sigma\p_s\eta\in L^2_{loc}([0,\infty); H^1(\Omega))^d$.
\item[(ii)] The pair $(\eta, \sigma)$ satisfies
for a.e. $(t,s)\in Q_\infty$
\begin{align} \p_t\eta(t,s) = \p_s(\sigma(t,s) \p_s\eta(t,s)) &+ g,\\
\sigma(t,s) \left(|\p_s\eta(t,s)|^2 -1\right ) &=0,\label{eq:defii}\\
|\p_s\eta(t,s)|&\leq 1,\label{eq:defii2}
\end{align}
and the initial/boundary conditions
\begin{align*}
& \eta(t,1)=0\text{ for all } t,\  \eta(0,s)=\eta_0(s)  \text{ for all } s,\ \sigma(t,0)=0  \text{ for a.e. } t.
\end{align*}
\item[(iii)] The solution $\eta$ satisfies the energy dissipation inequality
\begin{equation} \label{e:edeq}
\int_\Omega |\p_t\eta(t,s)|^2 ds\leq \int_\Omega g\cdot \p_t\eta(t,s) ds
\end{equation} for a.e. $t\in (0,\infty)$.
\end{itemize}
\end{defi}
We point out that Remark \ref{rmk:ib} concerning the initial and boundary conditions applies to this definition.
\begin{rmk}\label{rmk:gene_sol}
It is not hard to see that if $(\eta,\sigma)$ is a $C^2$ regular solution pair of \eqref{eq:whip}, \eqref{eq:bdry2}, then it is also a generalized solution in the sense of Definition \ref{defi:limit_sol}; in particular, \eqref{eq:defii2}
 and \eqref{e:edeq} become strict equalities. On the other hand we claim that any  generalized solution $(\eta,\sigma)$ with $\eta\in C^1(\overline{Q_\infty})\cap C^2(Q_\infty)$ and $|\p_s\eta_0|=1$ is a solution to \eqref{eq:whip}, \eqref{eq:bdry2}. It is an \emph{open problem} whether there exist generalized solutions which violate the constraint $|\p_s\eta|=1$ on a subset of $Q_\infty$ of positive Lebesgue measure.

To prove the claim it suffices to show that the open set $U:=\{(t,s)\in Q_\infty: |\p_s\eta(t,s)|<1\}$ is empty. Suppose not, then $\sigma=0$ a.e. in $U$ due to the constraint $\sigma(|\p_s\eta|-1)=0$. This implies that $\p_t\eta=g$ hence $\p_{st}\eta=0$ in $U$. For each $(t_0,s_0)\in U$, let $t_1=\inf\{t\geq 0: (t,t_0)\times\{ s_0\}\subset U\}$.
If $t_1=0$ then
\begin{equation}\label{e:ca2}
|\p_s\eta(t_1,s_0)|=1
\end{equation}
due to our assumption about $\eta_0$, and if $t_1>0$ then \eqref{e:ca2} also holds by the continuity of $\p_s\eta$. From $\p_{st}\eta = 0$ in $U$ and the up to the boundary $C^1$ continuity of $\eta$, we deduce that \[|\p_s\eta(t_0,s_0)|= |\p_s\eta(t_1,s_0)|= 1,\] arriving at a contradiction.
\end{rmk}

The next theorem provides the global existence of generalized solutions.

\begin{thm}\label{thm:limit} For every $\eta_0\in W^{1,\infty}(\Omega)^d$ with $\eta_0(1)=0$ and $|\p_s\eta_0(s)|\leq 1$ for a.e. $s\in \Omega$, there exists a generalized solution to \eqref{eq:whip}, \eqref{eq:bdry2} in $Q_\infty$. Moreover, those solutions satisfy $\sigma(t,s)\geq 0$ for almost every $(t,s)\in Q_\infty$.
\end{thm}
\begin{proof}
Let $T_k$ be a sequence of time with $T_k\rightarrow \infty$ as $k\rightarrow \infty$. For each $T_k$ fixed let $\{(\eta^\epsilon_k, \sigma^\epsilon_k)\}_\epsilon$ be solutions to the approximation problems \eqref{eq:approx_whip}-\eqref{eq:bdry} in $[0,T_k)\times (0,1)$. Remember that by Remark \ref{rmk:initial} we could approximate the initial datum by smooth ones with uniformly bounded energies. By Proposition \ref{prop:uniform} and a standard diagonal argument one can obtain a subsequence $(\eta_j,\sigma_j):=(\eta^{\epsilon_{k_j}}_{k_j}, \sigma^{\epsilon_{k_j}}_{k_j})$ and $(\eta,\sigma)$, such that $\eta_j\rightarrow \eta$ weakly* in $L^\infty_{loc}([0,\infty);W^{1,\infty}(\Omega))$, strongly in $L^2_{loc}([0,\infty);L^2(\Omega))$, $\p_t\eta_j\rightarrow \p_t\eta$ weakly in $L^2_{loc}([0,\infty);L^2(\Omega))$ and $\sigma_j\rightarrow \sigma$ weakly in $L^2_{loc}([0,\infty);H^1(\Omega))$. Furthermore, by (ii) of Proposition \ref{prop:uniform} and by Proposition \ref{prop:lim_eta}, the limit $(\eta, \sigma)$ is a generalized solution in the sense of (ii) of Definition \ref{defi:limit_sol} with a.e. non-negative $\sigma$.

Now we show $(\eta,\sigma)$ satisfies (iii) of Definition \ref{defi:limit_sol}. Indeed, multiplying $\p_t\eta_j \varphi$, $\varphi=\varphi(t)\in C_c^\infty((0,\infty))$, to the both sides of the equation of $\eta_j$  gives
\begin{align*}
\int_{Q_\infty} |\p_t\eta_j|^2\varphi\ ds dt = \int_{Q_\infty} \p_sG_j(\p_s\eta_j)\p_t\eta_j \varphi\ dsdt+\int_{Q_\infty} g\cdot \p_t\eta_j \varphi\ dsdt.
\end{align*}
Here $G_j=G^{\epsilon_{k_j}}$. Thus we only need to show that $\int_{Q_\infty} \p_sG_j(\p_s\eta_j)\cdot \p_t\eta_j \varphi \ dsdt\rightarrow 0$ as $j\rightarrow \infty$. After one integration by parts in the space variable (and recalling $\kappa_j=G_j(\p_s\eta_j)$) it suffices to show that
\begin{equation}\label{eq:ortho}
\int_{Q_\infty} \kappa_j\cdot \p_{ts}\eta_j\varphi\ dsdt\rightarrow 0.
\end{equation}
For this using $\p_s\eta_j= \epsilon_{k_j} \kappa_j+\frac{\kappa_j}{\sqrt{\epsilon_{k_j}+|\kappa_j|^2}}$ we have
\begin{equation}\label{eq:G}
\begin{split}
\kappa_j\cdot \p_{st}\eta_j&=\kappa_j\cdot \left(\epsilon_{k_j} \kappa_j+\frac{\kappa_j}{\sqrt{\epsilon_{k_j}+|\kappa_j|^2}}\right)_t\\
&=\epsilon_{k_j} \kappa_j \p_t\kappa_j +\kappa_j\cdot\left(\frac{\p_t\kappa_j}{\sqrt{\epsilon_{k_j}+|\kappa_j|^2}}-\frac{\kappa_j(\kappa_j\cdot \p_t\kappa_j)}{(\sqrt{\epsilon_{k_j}+|\kappa_j|^2})^3}\right)\\
&=\epsilon_{k_j} \kappa_j \p_t\kappa_j +\epsilon_{k_j} \frac{\kappa_j \cdot \p_t\kappa_j}{(\sqrt{\epsilon_{k_j}+|\kappa_j|^2})^3}\\
&=\epsilon_{k_j} \frac{d}{dt}\left(\frac{|\kappa_j|^2}{2}-\frac{1}{\sqrt{\epsilon_{k_j}+|\kappa_j|^2}}\right).
\end{split}
\end{equation}
Thus
\begin{align*}
\int_{Q_\infty} \kappa_j\cdot \p_{ts}\eta_j\varphi\ dsdt=-\epsilon_{k_j}\int_{Q_\infty} \left(\frac{|\kappa_j|^2}{2}-\frac{1}{\sqrt{\epsilon_{k_j}+|\kappa_j|^2}}\right) \frac{d}{dt}\varphi \ dsdt.
\end{align*}
From (i) in Proposition~\ref{prop:uniform}, $|\kappa_j|$ is uniformly bounded in $L^2(0,T; H^1(\Omega))$ where $T$ is such that $\varphi$ vanishes outside of $(0,T)$. Passing to the limit $j\rightarrow \infty$ we obtain \eqref{eq:ortho}.
\end{proof}

\section{Exponential decay}
\label{s:exp}
In this section we show that the relative potential energy decays along the trajectories of the generalized solutions exponentially fast.

We start by recalling the potential energy
$$E(t):=E(\eta(t,\cdot))=\int_0^1 (-g)\cdot \eta(t,s) ds.$$
We note that given a generalized solution the associated energy may not decrease along the trajectories. Indeed, let $(\eta,\sigma)$ be a generalized solution in the sense of Definition \ref{defi:limit_sol}. Then for $0\leq t_1<t_2<\infty$,
\begin{align*}
E(t_2)-E(t_1)&=\int_{t_1}^{t_2}\int_0^1(-g)\cdot \p_t \eta(t,s) dsdt\\
&=\int_{t_1}^{t_2}\int_0^1(-g)\cdot \p_s(\sigma\p_s\eta) +(-g)\cdot g dsdt\\
&=\int_{t_1}^{t_2}\left(\sigma(t,1)(-g)\cdot \p_s\eta(t,1) -1\right) dt.
\end{align*}
Since for generalized solutions we do not have any control on $\sigma(t,1)$, it is not clear whether the integral is nonpositive or not. On the other side, as already seen in Section \ref{s:gf}, if $(\eta,\sigma)$ is $C^2$-regular, using the equation of $\sigma$ it is possible to show that $\sigma(t,1)\leq 1$ for all $t$, which implies $E(t_2)\leq E(t_1)$.\\

For any generalized solution $(\eta, \sigma)$ with $\sigma\geq 0$ almost everywhere, we will show that the relative energy $\tilde{E}(t):=E(t)-E(\eta_\infty)$, where $(\eta_\infty, \sigma_\infty)(s)=((1-s)g, s)$ is the downwards vertical stationary solution, has an  upper bound which decays exponentially fast to zero as $t\rightarrow \infty$. This together with the nonnegativity of the relative energy (cf. Remark \ref{rmk:nonnegative}) implies the convergence of $\eta(t,\cdot)$ to $\eta_\infty$ in $L^2(\Omega)$ with an exponential convergence rate (cf. Corollary \ref{cor:conv_eta}).

Before we state our exponential convergence theorem we remark that using an integration by parts one can rewrite the energy as
$$E(t)=\int_0^1 s g\cdot \p_s\eta(t,s) ds.$$
A direct computation gives that $E(\eta_\infty)\equiv -\frac{1}{2}$.
The main result of this section is the following:
\begin{thm}\label{thm:exp}
Let $(\eta, \sigma)$ be a generalized solution in the sense of Definition \ref{defi:limit_sol}. Let $\tilde{E}(t):=E(t)-E(\eta_\infty)$ denote the relative energy. Assume that $\sigma\geq 0$ almost everywhere in $Q_\infty$. Then there exists a universal constant $c_0>0$ such that
\begin{align}\label{eq:exp}
\tilde{E}(t)\leq e^{-c_0t} \tilde{E}(0),\quad t\in [0,\infty).
\end{align}
\end{thm}
\begin{proof}
Using that
$\p_s\eta_\infty=-g$, one can rewrite $E(t)=-\int_0^1s\p_s\eta_\infty\cdot \p_s\eta ds$.
Thus
\begin{align*}
\tilde{E}(t)&=-\int_0^1 s\p_s\eta_\infty \cdot \p_s(\eta-\eta_\infty) ds\\
&=\int_0^1s|\p_s(\eta-\eta_\infty)|^2 ds-\int_0^1s\p_s \eta \cdot \p_s(\eta-\eta_\infty)ds
\end{align*}
Noting that
\begin{align*}
\int_0^1s\p_s \eta \cdot \p_s(\eta-\eta_\infty)=\int_0^1s\left(|\p_s\eta|^2-1\right)ds +\tilde{E}(t),
\end{align*}
we obtain an equivalent expression of the relative energy
\begin{align}\label{eq:energy}
\tilde{E}(t)=\frac{1}{2}\int_0^1s|\p_s(\eta-\eta_\infty)|^2 ds-\frac{1}{2}\int_0^1s\left(|\p_s\eta|^2-1\right)ds.
\end{align}
\emph{Claim:} There exists a universal $\bar{c_0}>0$ such that
\begin{equation}\label{eq:exp_decay}
\tilde{E}(t)\leq \bar{c_0} \int_{0}^1|\p_t\eta(t,s)|^2 ds
\end{equation}
in the sense of distributions.\\
\emph{Proof of the claim:} We first prove a lower bound for $\int_0^1|\p_t\eta|^2\ ds$. Indeed, using the equation of $\eta$ and Hardy's inequality one has for a.e. $t\in (0,\infty)$,
\begin{equation}\label{eq:eta_t}
\begin{split}
\int_0^1|\p_t\eta|^2 ds &= \int_0^1|\p_s\kappa+g|^2 ds= \int_0^1|\p_s\kappa-\p_s\kappa_\infty|^2 ds\\
&\geq \bar C \int_0^1 s^{-1}|\kappa-\kappa_\infty|^2 ds=\bar C\int_0^1s^{-1}|\sigma\p_s\eta- s\p_s\eta_\infty|^2 ds.
\end{split}
\end{equation}
Here $\bar C$ is a universal constant which is independent of $t$, and $\kappa_\infty:=\sigma_\infty\p_s \eta_\infty=-gs$. This implies that for any $\phi=\phi(t)\in C^\infty_c((0,\infty))$, $\phi\geq 0$,
\begin{equation}\label{eq:eta_t22}
\int_{Q_\infty} |\p_t\eta|^2\phi\ dsdt\geq \bar C\int_{Q_\infty} s^{-1}|\sigma\p_s\eta- s\p_s\eta_\infty|^2 \phi\ dsdt.
\end{equation}
We take the precise representatives of $\p_s\eta$ and $\sigma$ (cf. \cite[Section 1.7.1]{EG}) and define (still use the same notation for the representatives)
\begin{equation}\label{eq:Omega}
\begin{split}
\Omega_1&:=\{(t,s)\in Q_\infty: |\p_s\eta(t,s)|=1\},\\
\Omega_{2,1}&:=\{(t,s)\in Q_\infty: |\p_s\eta(t,s)|\neq 1, \ \sigma(t,s)=0\},\\
\Omega_{2,2}&:=\{(t,s)\in Q_\infty: |\p_s\eta(t,s)|\neq 1, \ \sigma(t,s)\neq 0\}.
\end{split}
\end{equation}
Using that $\sigma=0$ in $\Omega_{2,1}$ we furthermore obtain from \eqref{eq:eta_t22} that
\begin{equation}\label{eq:eta_t2}
\int_{Q_\infty} |\p_t\eta|^2\phi(t)\ dsdt\geq \bar C\int_{\Omega_{2,1}} s\phi(t)\ dsdt.
\end{equation}

Next we prove an upper bound for the relative energy $\tilde{E}(t)$. By \eqref{eq:energy},
\begin{align*}
\int\tilde{E}(t)\phi(t) dt&=\frac{1}{2}\int_{\Omega_1\cup\Omega_{2,1}} s|\p_s(\eta-\eta_\infty)|^2\phi\ ds dt+\frac{1}{2}\int_{\Omega_{2,2}} s|\p_s(\eta-\eta_\infty)|^2\phi\ ds dt\\
&-\frac{1}{2}\int_{\Omega_{2,1}} s\left(|\p_s\eta|^2-1\right)\phi\ dsdt-\frac{1}{2}\int_{\Omega_{2,2}} s\left(|\p_s\eta|^2-1\right)\phi\ dsdt
\end{align*}
Note that by \eqref{eq:defii} in Definition \ref{defi:limit_sol} we have
$\left|\Omega_{2,2}\right| =0$.
This together with $\eta\in L^\infty_{loc}([0,\infty);W^{1,\infty}(\Omega))^d$ (cf. (i) of Definition \ref{defi:limit_sol}) yields that the integrals over $\Omega_{2,2}$ are zero, i.e.
\begin{align*}
\int_{\Omega_{2,2}} s|\p_s(\eta-\eta_\infty)|^2\phi\ ds dt, \ \int_{\Omega_{2,2}} s\left(|\p_s\eta|^2-1\right)\phi\ dsdt=0.
\end{align*}
To estimate the integrals over  $\Omega_{2,1}$, we note that by \eqref{eq:defii2} of Definition \ref{defi:limit_sol},\\
 $\left|\Omega_{2,1}\cap \{(t,s)\in Q_T:|\p_s\eta(t,s)|>1\}\right| =0$. This gives
\begin{align*}
&\frac{1}{2}\int_{\Omega_{2,1}} s|\p_s(\eta-\eta_\infty)|^2\phi\ ds dt-\frac{1}{2}\int_{\Omega_{2,1}} s\left(|\p_s\eta|^2-1\right)\phi\ dsdt\\
= &\int_{\Omega_{2,1}} s\left(1-\p_s\eta\cdot \p_s\eta_\infty\right)\phi\ ds dt\\
\leq &\int_{\Omega_{2,1}}2 s\phi\ dsdt.
\end{align*}
For the integral over $\Omega_1$ we use Cauchy-Schwartz to obtain
\begin{align*}
&\quad \frac{1}{2}\int_{\Omega_{1}} s|\p_s(\eta-\eta_\infty)|^2\phi\ dsdt=\frac{1}{2}\int_{\Omega_1} s^{-1}|\sigma\p_s\eta-s\p_s\eta_\infty+(s-\sigma)\p_s\eta|^2\phi\ dsdt\\
&\leq\int_{Q_\infty} s^{-1}|\sigma\p_s\eta-s\p_s\eta_\infty|^2 \phi\ dsdt+\int_{\Omega_1} s^{-1}|s-\sigma|^2\phi\ dsdt.
\end{align*}
Since $|\p_s\eta|=1$ almost everywhere in $\Omega_1$, then by the triangle inequality and $\sigma\geq 0$ we have $|s-\sigma|=\left||s\p_s\eta_\infty|-|\sigma\p_s\eta|\right|\leq \left|\sigma\p_s\eta-s\p_s\eta_\infty\right|$ almost everywhere in $\Omega_1$. Thus
\begin{align*}
\int_{\Omega_1} s^{-1}|s-\sigma|^2\phi\ dsdt\leq \int_{\Omega_1} s^{-1}|\sigma\p_s\eta-s\p_s\eta_\infty|^2 \phi\ dsdt.
\end{align*}
Therefore combining all together we have, for all $\phi=\phi(t)\geq 0$ and $\phi\in C^\infty_c((0,\infty))$
\begin{align*}
\int_0^\infty \tilde{E}(t)\phi(t)\ dt\leq 2\int_{Q_\infty} s^{-1}|\sigma\p_s\eta-s\p_s\eta_\infty|^2\phi(t) \ dsdt+\int_{\Omega_{2,1}} 2s\phi(t)\ dsdt.
\end{align*}
Recalling the lower bound in \eqref{eq:eta_t22} and \eqref{eq:eta_t2} we thus obtain a universal $\bar{c_0}=4/\bar C>0$ such that
\begin{align*}
\int_0^\infty \tilde{E}(t)\phi(t)\ dt\leq \bar{c_0} \int_{Q_\infty}|\p_t\eta|^2\phi(t)\ dsdt.
\end{align*}
This completes the proof of the claim.

Now we use (iii) in Definition \ref{defi:limit_sol} to obtain the exponential decay of the relative energy. Indeed, given $(\eta, \sigma)$ a solution pair in the sense of Definition \ref{defi:limit_sol}, \eqref{eq:exp_decay}, (iii) of Definition \ref{defi:limit_sol} and an integration by parts yield
\begin{align*}
\int_0^\infty \tilde{E}(t)\phi(t)\ dt&\leq \bar{c_0}\int_0^\infty \int_0^1 g\cdot \p_t\eta (t,s)\phi(t)\ dsdt\\
&=\bar{c_0}\int_0^\infty \tilde{E}(t)\frac{d}{dt}\phi(t)\ dt.
\end{align*}
Denoting $c_0=\bar{c_0}^{-1}=\bar C/4$ we have
\begin{align*}
\int_0^\infty \tilde{E}(t)e^{c_0t}\frac{d}{dt} (e^{-c_0t}\phi(t))  dt\geq 0,
\end{align*}
for all $\phi=\phi(t)\geq 0$, $\phi\in C^{\infty}_c((0,\infty))$. This implies that $d(e^{c_0t}\tilde{E}(t))/dt$ is a nonpositive distribution. Since $t\mapsto \tilde{E}(t)$ is continuous, we have $\tilde{E}(t)\leq e^{-c_0t}\tilde{E}(0)$ for all $t\in [0,\infty)$.
\end{proof}

\begin{rmk}\label{rmk:nonnegative}
From the equivalent expression of the relative energy $\tilde{E}(t)$ in \eqref{eq:energy},  $|\p_s\eta|\leq 1$ a.e. (cf. \eqref{eq:defii2} in Definition \ref{defi:limit_sol}) and the continuity of $t\mapsto \tilde{E}(t)$, we immediately obtain that $\tilde{E}(t)\geq 0$ for all $t\in [0,\infty)$.

We also stress that the assumption $\sigma\geq 0$ a.e. in $Q_\infty$ in Theorem \ref{thm:exp} is not restrictive. On one hand, it is satisfied by the generalized solutions existing by Theorem \ref{thm:limit} for all Lipschitz initial data with $\eta_0(1)=0$ and $|\p_s\eta_0(s)|\leq 1$ for a.e. $s\in \Omega$. On the other hand, it automatically holds for any $C^2$-smooth solution $(\eta,\sigma)$ with $\alpha(\eta_0)\leq \pi/2$ (cf. Section \ref{s:gf}).
\end{rmk}

The exponential decay of the relative energy implies the exponential decay of the generalized solution to the stationary solution $\eta_\infty$ in $L^2(\Omega)$ as $t\rightarrow \infty$.

\begin{cor}\label{cor:conv_eta}
Let $(\eta,\sigma)$ be a generalized solution in the sense of Definition \ref{defi:limit_sol} with $\sigma\geq 0$. Let $\tilde{E}(t)=E(\eta(t,\cdot))-E(\eta_\infty)$ be the relative energy as in Theorem \ref{thm:exp}, where $(\eta_\infty,\sigma_\infty)=((1-s)g,s)$ is the stable stationary solution. Then there exist a universal constant $c_0>0$ such that for all $t\in [0,\infty)$,
\begin{align}
\left\|\eta(t,\cdot)-\eta_\infty(\cdot)\right\|^2_{L^2(\Omega)}&\leq (2c_0)^{-1}\tilde{E}(0)e^{-c_0t},\label{eq:exp_eta}\\
\int_t^\infty \left\|s^{-1/2}\left(\sigma(\tilde{t},\cdot)-\sigma_\infty(\cdot)\right)\right\|_{L^2(\Omega)}^2 \ d\tilde{t} &\leq (2c_0)^{-3/2} \tilde{E}(0)^{1/2}e^{-c_0 t/2}. \label{eq:exp_sigma}
\end{align}
\end{cor}
\begin{proof}
\emph{Proof of \eqref{eq:exp_eta}:}
By \eqref{eq:exp} for any nonnegative $\varphi=\varphi(t)\in C^\infty_c((0,\infty))$,
\begin{align*}
\int_0^\infty\tilde{E}(t)\varphi(t) dt\leq \int_0^\infty e^{-c_0t}\tilde{E}(0) \varphi(t) dt .
\end{align*}
Using the equivalent expression \eqref{eq:energy} for $\tilde{E}(t)$ we obtain
\begin{align*}
\int_{Q_\infty}s|\p_s(\eta-\eta_\infty)|^2 \varphi\ dsdt-\int_{Q_\infty} s\left(|\p_s\eta|^2-1\right)\varphi\ dsdt\leq  2 \int_0^\infty e^{-c_0t}\tilde{E}(0) \varphi\ dt.
\end{align*}
Since $\p_s\eta\in L^2_{loc}(Q_\infty)$ and $|\p_s\eta|\leq 1$ for almost every $(t,s)\in Q_\infty$ (cf. Definition \ref{defi:limit_sol}), we have
$
-\int_{Q_\infty}s\left(|\p_s\eta|^2-1\right)\varphi\ dsdt\geq 0.
$
Thus we immediately obtain the exponential decay of the Sobolev distance:
\begin{equation}\label{eq:sob}
\left\|\sqrt{s}(\p_s\eta-\p_s\eta_\infty)\right\|_{L^2(\Omega)}^2\leq 2\tilde{E}(0)e^{-c_0t} \text{ for almost every }t\in [0,\infty).
\end{equation}
To derive the decay for the $L^2$ distance we apply again the Hardy's inequality
\begin{align*}
\bar C\int_0^1\left|\eta(t,s)-\eta_\infty(s)\right|^2 ds\leq \int_0^1s|\p_s(\eta(t,s)-\eta_\infty(s))|^2 ds
\end{align*}
for almost every $t$. Here $\bar C=4c_0$ is the same as in the proof for Theorem \ref{thm:exp}. This together with \eqref{eq:sob} yields
\begin{align*}
\bar C\int_{Q_\infty}\left|\eta(t,s)-\eta_\infty(s)\right|^2 \varphi(t) \ dsdt\leq \int_0^\infty e^{-c_0t}\tilde{E}(0) \varphi(t)\ dt.
\end{align*}
Since this holds for arbitrary nonnegative $\varphi\in C_c^\infty((0,\infty))$ and since the map $t\mapsto \|\eta(t,\cdot)-\eta_\infty(\cdot)\|_{L^2(\Omega)}$ is continuous by virtue of \eqref{e:emb}, we conclude that
\begin{align*}
\bar C\|\eta(t,\cdot)-\eta_\infty(\cdot)\|^2_{L^2(\Omega)}\leq  2\tilde{E}(0)e^{-c_0t} \text{ for  all } t\in [0,\infty).
\end{align*}
Recalling that $\bar C=4c_0$ we obtain \eqref{eq:exp_eta}.\\

\emph{Proof of \eqref{eq:exp_sigma}:} Firstly, \eqref{e:edeq} in Definition \ref{defi:limit_sol} and the exponential decay of $\|\eta(t,\cdot)-\eta_\infty(\cdot)\|_{L^2(\Omega)}$ in \eqref{eq:exp_eta} yield that for all $T\geq 0$
\begin{align*}
\int_T^\infty \int_\Omega |\p_t\eta|^2 \ dsdt &\leq \int_T^\infty \int_\Omega g\cdot \p_t \eta \ dsdt=\int_\Omega g\cdot(\eta_\infty(s)-\eta(T,s)) \ ds\\
&\leq \|\eta(T,\cdot)-\eta_\infty(\cdot)\|_{L^2(\Omega)}\leq (2c_0)^{-1/2}\tilde{E}(0)^{1/2}e^{-c_0T/2}.
\end{align*}
Integrating \eqref{eq:eta_t} from $T$ to $\infty$ and then using the above estimate we obtain
\begin{align} \label{eq:sigma_limit}
\bar C\int_T^\infty \int_\Omega s^{-1}|\sigma\p_s\eta-\sigma_\infty\p_s\eta_\infty|^2 \ dsdt\leq \int_T^\infty \int_\Omega |\p_t\eta|^2 \ dsdt\leq (2c_0)^{-1/2}\tilde{E}(0)^{1/2}e^{-c_0T/2}.
\end{align}
We write the left side of \eqref{eq:sigma_limit} as a sum of integrals over $\Omega_1$, $\Omega_{2,1}$ and $\Omega_{2,2}$ and estimate them separately (cf. \eqref{eq:Omega} for the definition of those sets). In $\Omega_1$ we have $|\p_s\eta|=1$. Thus by the triangle inequality and $\sigma, \sigma_\infty\geq 0$ we have $|\sigma-\sigma_\infty|= \left||\sigma\p_s\eta|-|\sigma_\infty\p_s\eta_\infty\right||\leq |\sigma\p_s\eta-\sigma_\infty\p_s\eta_\infty|$ in $\Omega_1$. In $\Omega_{2,1}$ using $\sigma=0$ and the explicit expression of $(\sigma_\infty,\eta_\infty)$ we have $s^{-1}|\sigma\p_s\eta-\sigma_\infty\p_s\eta_\infty|^2=s^{-1}|\sigma_\infty\p_s\eta_\infty|^2=s=s^{-1}|\sigma-\sigma_\infty|^2$. Finally $|\Omega_{2,2}|=0$ by (ii) of Definition \ref{defi:limit_sol}. These observations together with \eqref{eq:sigma_limit} yield
\begin{align*}
\int_{T}^\infty \int_\Omega s^{-1}|\sigma-\sigma_\infty|^2 \ dsdt\leq (4c_0)^{-1}(2c_0)^{-1/2}\tilde{E}(0)^{1/2}e^{-c_0T/2}.
\end{align*}
This completes the proof for \eqref{eq:exp_sigma}.
\end{proof}

\section{Backward solutions and non-uniqueness}
\label{s:backward}
In this section we present some consequences of our previous results providing existence of solutions backwards in time and branching of (generalized) trajectories of our gradient flow.  \begin{defi}\label{defi:back_sol} Given an initial datum $\eta_0\in W^{1,\infty}(\Omega)^d$ with $\eta_0(1)=0$ and $|\p_s\eta_0(s)|\leq 1$ for almost every $s\in \Omega$, a pair $(\eta,\sigma)$ is generalized solution to \eqref{eq:whip}, \eqref{eq:bdry2} in $Q_{-\infty}:=(-\infty,0)\times \Omega$ if
\begin{itemize}
\item[(i)]  $\eta\in L^\infty_{loc}((-\infty,0]; W^{1,\infty}(\Omega))^d$, $\p_t\eta\in L^2_{loc}((-\infty,0]);L^2(\Omega))^d$, $\sigma\in L^2_{loc}((-\infty,0] ;H^1(\Omega))$ and $\sigma\p_s\eta\in L^2_{loc}((-\infty,0]; H^1(\Omega))^d$.
\item[(ii)] The pair $(\eta, \sigma)$ satisfies
for a.e. $(t,s)\in Q_{-\infty}$
\begin{align} \p_t\eta(t,s) = \p_s(\sigma(t,s) \p_s\eta(t,s)) &+ g,\\
\sigma(t,s) \left(|\p_s\eta(t,s)|^2 -1\right ) &=0,\label{eq:defiib}\\
|\p_s\eta(t,s)|&\leq 1,\label{eq:defii2b}
\end{align}
and the initial/boundary conditions
\begin{align*}
& \eta(t,1)=0\text{ for all } t,\  \eta(0,s)=\eta_0(s)  \text{ for all } s,\ \sigma(t,0)=0  \text{ for a.e. } t.
\end{align*}
\item[(iii)] The solution $\eta$ satisfies the energy dissipation inequality
\begin{equation} \label{e:edeqb}
\int_\Omega |\p_t\eta(t,s)|^2 ds\leq \int_\Omega g\cdot \p_t\eta(t,s) ds
\end{equation} for a.e. $t\in (-\infty,0)$.
\end{itemize}
\end{defi}
As the following proposition shows, the backward solutions can be easily constructed by considering the forward problem for the upward gravity $-g$, and reversing the time afterwards.
\begin{prop}\label{p:back} For every $\eta_0\in W^{1,\infty}(\Omega)^d$ with $\eta_0(1)=0$ and $|\p_s\eta_0(s)|\leq 1$ for a.e. $s\in \Omega$, there exists a generalized solution to \eqref{eq:whip}, \eqref{eq:bdry2} in $Q_{-\infty}$. The solution satisfies $\sigma(t,s)\leq 0$ a.e. in $Q_{-\infty}$. Set $\hat{E}(t):=E(\eta_{-\infty})-E(\eta(t,\cdot))$, where $(\eta_{-\infty}, \sigma_{-\infty})(s):=((s-1)g, -s)$ is the unstable upright stationary solution. Then \begin{align}\label{eq:expback}
0 \leq\hat{E}(t)\leq e^{c_0t} \hat{E}(0),\quad t\in (-\infty,0],
\end{align} where $c_0>0$ is the same as in Theorem \ref{thm:exp}.
\end{prop}
\begin{proof} By Theorem \ref{thm:limit}, there exists a generalized solution $(\eta^-,\sigma^-)$, $\sigma^-\geq 0$ a.e. in $Q_\infty$, to the problem \eqref{eq:whip}, \eqref{eq:bdry2} with $g$ replaced by $-g$. The corresponding potential energy  is $$E^-(\eta^-):=\int_0^1 g\cdot \eta^-\,ds.$$ Noting that $(\eta_{-\infty}, \sigma_{\infty})=((s-1)g, s)$ is the stable stationary solution the problem \eqref{eq:whip}, \eqref{eq:bdry2} with the reversed gravity $-g$, we introduce the relative energy $\tilde{E}^-(t):=E^-(\eta^-(t, \cdot))-E^-(\eta_{-\infty})$. By Theorem \ref{thm:exp} and Remark \ref{rmk:nonnegative}, this relative energy  satisfies \begin{align}\label{eq:expbackf}
0 \leq\tilde{E}^-(t)\leq e^{-c_0t} \tilde{E}^-(0),\quad t \geq 0.
\end{align} Define the pair of functions $(\eta,\sigma):Q_{-\infty}\to \R^d\times \R$ by $$(\eta,\sigma)(t,\cdot):=(\eta^-,-\sigma^-)(-t,\cdot).$$ This pair solves \eqref{eq:whip}, \eqref{eq:bdry2} in $Q_{-\infty}$, and $\sigma\leq 0$ a.e. in $Q_{-\infty}$. Observe that the dissipation inequality \eqref{e:edeqb} holds since \begin{equation*} \int_\Omega |\p_t\eta(t,s)|^2 ds=
\int_\Omega |\p_t\eta^-(-t,s)|^2 ds\leq \int_\Omega (-g)\cdot \p_t\eta^-(-t,s) ds=\int_\Omega g\cdot \p_t\eta(t,s) ds
\end{equation*} for a.e. $t< 0$. Finally, $$\hat E(t)=E(\eta_{-\infty})-E(\eta(t,\cdot))=E^-(\eta^-(-t,\cdot))-E^-(\eta_{-\infty})=\tilde{E}^-(-t),$$ so the exponential decay \eqref{eq:expback} follows from \eqref{eq:expbackf}. \end{proof} The next statement shows that the generalized solutions can be non-unique. The idea is to follow the backward trajectory for a while, and then start going forward.  Then we have two options: either to return by the same trajectory (with non-positive $\sigma$) or to choose the trajectory provided by  Theorem \ref{thm:limit} (with non-negative $\sigma$).

\begin{prop}\label{p:branch} Let $(\bar\eta, \bar\sigma)$ be any backward generalized solution provided by Proposition \ref{p:back}. Then there exists a null set $O\subset (0,\infty)$ so that for every initial datum of the form $\eta_0:=\bar\eta(-T)$, $T>0$, $T\not\in O$, there are at least two different generalized solutions on $Q_T$ emanating from $\eta_0$. \end{prop}
\begin{proof} Let us first check that Theorem \ref{thm:limit}  is applicable to $\eta_0$ defined in the statement of the proposition. Since $|\p_s\bar\eta(t,s)|^2\leq 1$ a.e. in $Q_{-\infty}$, we have $|\p_s\eta_0(s)|^2=|\p_s\bar\eta(-T,s)|^2\leq 1$ a.e. in $\Omega$ for a.e. $T>0$. Moreover, without loss of generality, $\eta_0(1)=0$. By Theorem \ref{thm:limit} there exists a generalized solution $(\eta^+,\sigma^+)$ on $Q_{\infty}$ emerging from $\eta_0$, and $\sigma^+\geq 0$ almost everywhere. On the other hand, consider the pair $(\eta^-,\sigma^-):Q_T\to \R^d\times \R$, $(\eta^-,\sigma^-)(t,\cdot):=(\bar\eta,\bar\sigma)(t-T,\cdot)$. Then $(\eta^-,\sigma^-)$ is a generalized solution to \eqref{eq:whip}, \eqref{eq:bdry2} on $Q_{T}$ originating from $\eta_0$, and $\sigma^-\leq 0$ a.e. in $Q_T$. If it were that $(\eta^+,\sigma^+)=(\eta^-,\sigma^-)$ a.e. in $Q_T$, then $\sigma_+$ would vanish a.e. in $Q_T$, whence $\eta_+=\eta_0+gt$, which would violate the first boundary condition in \eqref{eq:bdry2}. \end{proof}

\begin{rmk} It is clear that a similar non-uniqueness property holds for the backward solutions.  \end{rmk}

\begin{rmk} \label{udwhip} Proposition \ref{p:branch} applies to the case $(\bar\eta, \bar\sigma)(t,s)=(\eta_{-\infty}, \sigma_{-\infty})(s)=((s-1)g, -s)$, showing that there is a non-stationary solution which originates from the stationary upright solution $\eta_{-\infty}$. In \cite[Section 5.2]{JDE17}, a very similar phenomenon was observed for the Young measure solutions to \eqref{eq:geodgrav}, \eqref{e:bkg}.   Actually, we can provide two explicit non-stationary generalized  solutions emanating from the upright string $\eta_{-\infty}$. The first one is \begin{equation*}
\begin{split}
\eta_1=\begin{cases}
(-1+t+s)g,& s\leq 1-\frac t 2\\
(1-s)g,& s> 1-\frac t 2
\end{cases}
\end{split}
\end{equation*} with
\begin{equation*}
\begin{split}
\sigma_1=\begin{cases}
0,& s\leq 1-\frac t 2\\
s+\frac t 2 -1,& s> 1-\frac t 2.
\end{cases}
\end{split}
\end{equation*} The second one is \begin{equation*}
\begin{split}
\eta_2=\begin{cases}
(-1+t-s)g,& s< \frac t 2\\
(s-1)g,& s\geq \frac t 2
\end{cases}
\end{split}
\end{equation*} with
\begin{equation*}
\begin{split}
\sigma_2=\begin{cases}
0,& s< \frac t 2\\
-s+\frac t 2,& s\geq \frac t 2.
\end{cases}
\end{split}
\end{equation*} Here we implicitly assumed that $t\leq 2$ because afterwards these solutions reach the stable equilibrium. It is easy to check that both pairs $(\eta_i,\sigma_i)$, $i=1,2$, satisfy Definition \ref{defi:limit_sol}. In both cases $\p_s \eta_i$ have discontinuities. We also observe that $\sigma_1$ is non-negative and continuous, but does not coincide with $\sigma_{-\infty}$ at $t=0$. Moreover, the orthogonal projection $P_{\eta_1(0)}g=0$, but for positive times it formally jumps to \begin{equation*}
\begin{split}
P_{\eta_1(t)}g=\p_t\eta_1(t)=\begin{cases}
g,& s< 1-\frac t 2\\
0,& s> 1-\frac t 2.
\end{cases}
\end{split}
\end{equation*} On the other hand, $\sigma_2$ is non-positive, continuous, and agrees with $\sigma_{-\infty}$ at $t=0$. Furthermore, the orthogonal projection \begin{equation*}
\begin{split}
P_{\eta_2(t)}g=\p_t\eta_2(t)=\begin{cases}
g,& s< \frac t 2\\
0,& s> \frac t 2
\end{cases}
\end{split}
\end{equation*} is continuous with respect to time in the $L^2$-topology. There are large zones where $\sigma_i\equiv 0$, but the strong constraint $|\p_s\eta_i|=1$ still holds almost everywhere. \end{rmk}

\subsection*{Acknowledgment} The idea of this paper originated from conversations of the second author with Yann Brenier during a stay at ESI in Vienna. He would like to thank
 Yann Brenier for the inspiring discussions, Ulisse Stefanelli for the invitation to the thematic program "Nonlinear Flows" at ESI, and ESI for hospitality. The research was supported  by CMUC (UID/MAT/00324/2019), funded by the Portuguese government through FCT and co-funded by the ERDF through PT2020, and by FCT through projects TUBITAK/0005/2014 and PTDC/MAT-PUR/28686/2017. 

\def\cprime{$'$}

\end{document}